\documentclass[reqno, 12pt]{amsart}
\usepackage{amssymb,amsmath,amsfonts}
\usepackage[margin=1in]{geometry}
\usepackage{dsfont}
\usepackage{color}
\usepackage{graphicx}
\usepackage{latexsym}
\usepackage{amsmath}
\usepackage{amssymb}
\usepackage{graphics}
\usepackage[dvips]{epsfig}
\usepackage{mathrsfs}
\usepackage{mathtools}
\usepackage{leqno}
\usepackage{stmaryrd,cite}

\usepackage{cases}
\usepackage{enumerate}
\usepackage[usenames,dvipsnames,svgnames]{xcolor}
\definecolor{refkey}{rgb}{1,0,0.5}
\definecolor{labelkey}{rgb}{0,0.4,1}
\usepackage{todonotes}
\usepackage{marginnote}

\presetkeys{todonotes}{fancyline, color=green!40}{}

\makeatletter
\renewcommand{\@todonotes@drawMarginNoteWithLine}{%
	\begin{tikzpicture}[remember picture, overlay, baseline=-0.75ex]%
	\node [coordinate] (inText) {};%
	\end{tikzpicture}%
	\marginnote[{
		\@todonotes@drawMarginNote%
		\@todonotes@drawLineToLeftMargin%
	}]{
		\@todonotes@drawMarginNote%
		\@todonotes@drawLineToRightMargin%
	}%
}
\makeatother
\usepackage{listings}
\usepackage{ifpdf}
\ifpdf
\usepackage[CJKbookmarks=true,
         hyperindex=true,
         pdfstartview=FitH,
         bookmarksnumbered=true,
         bookmarksopen=true,
         colorlinks=true,
         citecolor=blue,
         linkcolor=blue,
         urlcolor=blue,
         pdfborder=001,
         pdfauthor={},
         pdftitle={},
         pdfkeywords={},
         ]{hyperref}
\else
\usepackage[
         hypertex,
         hyperindex,
         linkcolor=blue,
         unicode,
         citecolor=blue%
         ]{hyperref}
\fi
\usepackage{cleveref}

\allowdisplaybreaks

\numberwithin{equation}{section}
\newtheorem{thm}{Theorem}[section]

\newtheorem{lem}[thm]{Lemma}
\newtheorem{prop}[thm]{Proposition}
\newtheorem{rmk}[thm]{Remark}
\newtheorem{defn}[thm]{Definition}
\newtheorem{app}[thm]{Appendix}

\newcommand{\pl}{\partial}

\newcommand{\be}{\begin{equation}}
\newcommand{\ee}{\end{equation}}

\newcommand{\bee}{\begin{equation*}}
\newcommand{\eee}{\end{equation*}}

\newcommand{\bse}{\begin{subequations}}
\newcommand{\ese}{\end{subequations}}

\newcommand{\bs}{\begin{split}}
\newcommand{\es}{\end{split}}

\begin{document}

\author{Hairong Liu$^{1}$}\thanks{$^{1}$School of Mathematics and Statistics, Nanjing University of Science and Technology, Nanjing, 210094, P.R China
E-mail : hrliu@njust.edu.cn}

\author{Tao Luo$^{2}$}\thanks{$^{2}$Department of Mathematics, City University of Hong Kong, 83 Tat Chee Avenue, Kowloon Tong, Hong Kong.
E-mail: taoluo@cityu.edu.hk}

\author{Hua Zhong$^{3}$}\thanks{$^{3}$School of Mathematics, Southwest Jiaotong University, Chengdu 611756, P.R. China.
E-mail: huazhong@swjtu.edu.cn}

\title[] {Strong solutions to  the 3-D compressible  MHD equations with density-dependent viscosities in exterior domains with far-field vacuum}

\begin{abstract}

This paper investigates the existence and uniqueness of local strong solutions to the three-dimensional compressible magnetohydrodynamic (MHD) equations with density-dependent viscosities in an exterior domain. The system models the dynamics of electrically conducting fluids, such as plasmas, and incorporates the effects of magnetic fields on fluid motion. We focus on the case where the viscosity coefficients are proportional to the fluid density, and the far-field density approaches vacuum. By introducing a reformulation of the problem using new variables to handle the degeneracy near vacuum, we establish the local well-posedness of strong solutions for arbitrarily large initial data, even in the presence of far-field vacuum. Our analysis employs energy estimates, elliptic regularity theory, and a careful treatment of the Navier-slip boundary conditions for the velocity and perfect conductivity conditions for the magnetic field. To the best of our knowledge,   such results are not available even  for the Cauchy problem to the 3-D compressible  MHD equations with degenerate viscosities. \\
 \noindent {\bf Keywords}: 3-D magnetohydrodynamic equations,
Far field vacuum, Density-dependent viscosities, Navier-slip boundary condition, Exterior domain.\\

 \noindent {\bf AMS Subject Classifications.} 35Q30, 76W05
\end{abstract}

\maketitle

\section{Introduction and Main Theorems}
We consider the three-dimensional compressible magnetohydrodynamic (MHD) equations:
\begin{equation}\label{prob1}
	\left\{
	\begin{array}{llll}
		\rho_{t}+\mbox{div}(\rho u)=0,\\[2mm]
		\rho\big(u_{t}+u\cdot\nabla u\big)+\nabla P=\mbox{div}\mathbb{T}+\mbox{curl}H\times H,\\[2mm]
		H_{t}- \eta \Delta H=\mbox{curl}(u\times H),\quad \quad \mbox{div}H=0,
	\end{array}
	\right.
\end{equation}
where $\rho$, $u=(u_1,u_2,u_3)$, $H=(H_1,H_2,H_3)$ denote the density, the velocity and the magnetic field,  respectively.
The constant $\eta>0$ is the magnetic diffusivity
which describes a magnetic diffusion coefficient of the magnetic field.
For polytropic gases, the constitutive relation is given by
\begin{equation}\label{P}
	P=A\rho^{\gamma},\quad A>0,\quad \gamma>1,
\end{equation}
where $A$ is an entropy constant and $\gamma$ is the adiabatic exponent.
$\mathbb{T}$ denotes the viscous stress tensor with the form
\begin{equation}\label{T}
	\mathbb{T}=\mu(\rho)\left(\nabla u+(\nabla u)^T\right)+\lambda(\rho)\mbox{div}u\mathbb{I}_{3}
\end{equation}
where $\mathbb{I}_3$ is the $3\times3$ identity matrix.  The viscosity coefficients $\mu$, $\lambda$ are the functions of $\rho$ and could be expressed as
\begin{equation}\label{T}
	\mu(\rho)=\alpha \rho^{\delta},\quad \lambda(\rho)=\beta\rho^{\delta},
\end{equation}
for some constant $\delta\geq0$, where $\mu(\rho)$ is the shear viscosity coefficient, $\lambda(\rho)+\frac{2}{3}\mu(\rho)$ is the bulk viscosity coefficient, $\alpha$ and $\beta$ are both constants satisfying
\begin{equation}\label{alpha}
	\alpha>0,\quad 2\alpha+3\beta\geq0,
\end{equation}
due to physical realities.  In this work, we focus on the case $\delta=1$ in \eqref{T}. \\

\noindent{ \bf Problem Setting}: \\

Let $U$ be a simply connected bounded smooth domain in $\mathbb{R}^{3}$, and $\Omega\equiv\mathbb{R}^{3}\backslash \bar{U}$ be the exterior domain.  We study the initial-boundary value problem of (\ref{prob1}) in $\Omega$ with:
\begin{itemize}
 \item Initial conditions
\begin{equation}\label{initial}
	(\rho, u,  H)|_{t=0}=(\rho_0(x), u_0(x),   H_0(x))
\end{equation}
with $\rho_0(x)>0$.
\item Navier-slip boundary conditions on $u$:
\begin{equation}\label{boundary2-1}
	u\cdot n=0,\quad 2 \left(S(u)\cdot n\right)_{\tau}=-(\vartheta u)_{\tau},\quad \mbox{on}\quad \partial\Omega,
\end{equation}
where $n$ stands for the outward unit normal to $\partial\Omega$, $S$ is the strain tensor,
\begin{equation*}
	S(u)=\frac{1}{2}\left(\nabla u+(\nabla u)^{T}\right),
\end{equation*}
and for some vector field $v$ on $\partial\Omega$, $v_{\tau}$ stands for the tangential part of $v$,
\begin{equation*}
	v_{\tau}=v-(v\cdot n)n.
\end{equation*}
These conditions,  first introduced by Navier in \cite{Navierslip},  models the situation that there is a stagnant layer of fluid close to the wall allowing a fluid to slip and the slip velocity is proportional to the shear stress. Here $\vartheta$ is the scalar friction function to measure the tendency of the fluid to slip on the boundary. The boundary condition (\ref{boundary2-1}) is equivalent to the following conditions in the sense of  the distribution
\begin{equation}\label{boundary2-2}
	u\cdot n=0,\quad \mbox{curl}u\times n= \left(2S(n)-\vartheta \mathbb{I}_3 \right)u=:-K(x)u,\quad \mbox{on}\quad \partial\Omega,
\end{equation}
where $K(x)$ is a $3\times3$ symmetric matrix.

\item Perfect conductivity conditions for $H$:
\begin{equation}\label{boundary3}
	\quad H\cdot n=0,\quad \mbox{curl}H\times n=0, \quad \mbox{on} \quad \partial\Omega.
\end{equation}
\item Far field behavior:
\begin{equation}\label{1.8}
	(\rho,u,H)(x,t)\rightarrow(0,0,0),\quad\text{as}\quad |x|\rightarrow\infty.
\end{equation}
Such kinds of far field behavior occurs naturally under some physical
assumptions on (\ref{prob1})'s solutions, such as finite total mass and total energy.

\end{itemize}
Moreover, we assume that the initial data satisfy  the compatibility condition  with the boundary condition.\\

\noindent {\bf Background}: \\

When $H = 0$, the MHD system reduces to compressible
Navier-Stokes equations.
For the constant viscous
fluid, there is a lot of literature on the well-posedness of classical solutions to isentropic compressible
Navier-Stokes equations.
When $\inf_{x}\rho_0>0$, it is well-known that the local existence of classical solutions for (\ref{prob1})-(\ref{1.8}) can be obtained by a standard Banach fixed point argument, see  Nash \cite{nash}.  And this method has been extended to be a global one by Matsumura-Nishida \cite{Matsumura1980} for initial data close to a non-vacuum equilibrium in some Sobolev space $H^s$ $(s > \frac{5}{2})$. However, when the density function connects to vacuum continuously, this approach is invalid because of $\inf_{x}\rho_0=0$, which occurs when some physical requirements are imposed, such as finite total initial mass, finite total initial energy, or vacuum appearing locally in some open sets.  Moreover, the degeneracy caused by the initial vacuum will cause some difficulties in the regularity estimate due to the less regularizing effect of the viscosity  on solutions.  When $\inf\rho_0(x)=0$, the local-in-time
well-posedness of strong solutions with vacuum was firstly solved by Cho-Choe-Kim  \cite{Cho2004} and Cho-Kim  \cite{Cho2006}
in $\mathbb{R}^3$, where they introduced an initial compatibility condition to compensate the lack of a positive lower
bound of density.
Later, Huang-Li-Xin \cite{HLX2012} extended the local existence to a global one
with smooth initial data that are of small energy but possibly
large oscillations in $\mathbb{R}^3$. Recently,  Cai-Li\cite{CL2021} established  the global existence in 3D bounded domains, and in exterior domain with Navier-slip boundary by  Cai-Li-Lv\cite{CLL2021}. Jiu-Li-Ye \cite{JLY2014} proved the global existence of classical solution with arbitrarily
large data and vacuum in $\mathbb{R}$.

When viscosity coefficients are density-dependent, the Navier-Stokes system has been received extensive
attentions in recent years, especially for the case with vacuum, where the well-posedness of solutions become
more challenging due to the degenerate viscosity.
Noted that for density-dependent viscosities this is $\delta>0$ in \eqref{T}, the approach of the constant viscosity case in \cite{Cho2006,Cho2004} doesn't work any more,  and the strong degeneracy of the momentum equations of Navier-Stokes equations near the vacuum creates serious difficulties for the well-posedness of both strong and weak solutions.
A remarkable discovery of a new mathematical entropy function has been found by
Bresch-Desjardins \cite{BD2003} for the viscosity satisfying some mathematical relation, which provides
additional regularity on some derivative of the density. This observation was applied widely
in proving the global existence of weak solutions with vacuum for Navier-Stokes equations
and some related models (see \cite{BVY2022},\cite{JX2008}, \cite{VY2016} and so on). For classical solutions,  when $\delta=1$, Li-Pan-Zhu \cite{LPZ20172D} obtained the local existence of 2-D classical solution with far field vacuum, which also applies to the 2-D shallow water equations.
When $1<\delta\leq\min{\{3,\frac{\gamma+1}{2}\}}$, Li-Pan-Zhu\cite{LPZ2019} established the
local existence of 3-D classical solutions with arbitrarily large data and vacuum.  Recently, Xin-Zhu\cite{XZ2021advance} have proved the global well-posedness of regular solutions with vacuum for a class of smooth initial data that
are of small density but possibly large velocities. This result is the first one on the global existence of smooth solution which have large velocities and contain vacuum state for such degenerate system in three space dimensions.
 When $0<\delta<1$,
Xin-Zhu \cite{XZ2021jmpa} identify a class of initial data admitting a local regular solution with far field vacuum and finite energy in some inhomogeneous Sobolev spaces by introducing some new variables and initial compatibility conditions for the Cauchy problem in 3-D space.  Cao-Li-Zhu \cite{CLZ2022} proved the global existence
of 1-D classical solution with large initial data. Some other interesting results and discussions can also be
found in \cite{CLZ2024}, \cite{DXZ2023},  \cite{GLZ2019}, \cite{GL2016}, \cite{LXZ2016}, \cite{YZ2002-1}, \cite{YZ2002} and the references therein.

 The system of  compressible  MHD equations have been also studied extensively by physicists and mathematicians because of its physical importance, rich and complex phenomena and mathematical challenges. In particular, the local and global well-posedness
 of solutions have been widely concerned.
When viscosity coefficients $\mu$ and $\lambda$ are both constants, we briefly review the some interesting results, for which one may refer to \cite{Kawashima1982} for one-dimensional case and\cite{CHS2023,CHS2024,Fan2009,  HW2010,HH2017,Kawashima1984,LY2011,LLZ2022,LXZ2013,PG2013,TG2017,ZZ2010, Zhu2015siam}  for higher dimensions. The local existence of strong solutions  for the Cauchy problem was obtained by
Fan-Yu \cite{Fan2009} with the initial density containing vacuum.
Tang-Gao \cite{TG2017} obtained the local strong solutions to the compressible MHD equations in a 3-D bounded domain with the Navier-slip condition.
For global existence, Kawashima \cite{Kawashima1984} firstly established the global smooth solutions to
the general electro-magneto-fluid equations in two dimensions with non-vacuum.
Hu-Wang \cite{HW2010} considered the initial-boundary
value problem in a bounded domain with large data and proved the existence and large-time behavior of
global weak solutions for $\gamma>\frac{3}{2}$.
Recently, Li-Xu-Zhang \cite{LXZ2013} confirmed the global well-posedness
of classical solution of the Cauchy problem for
initial data with small energy but possibly large oscillations,  where the density is allowed to contain vacuum states. Hong et al. \cite{HH2017} generalized the result for large initial data when $\gamma-1$ and $\eta^{-1}$ are suitably small.
The initial-boundary value problem with small energy can be found in recently work \cite{CHS2023} for bounded domain, and \cite{CHS2024} for  exterior domain.

For the case of density-dependent viscosity coefficients (i.e., $ \delta>0$ in \eqref{T}), the problem is much more challenging due to the degeneration near the vacuum and hence the obtained results are limited.
Huang-Shi-Sun\cite{HSS2019} studied  the equations of a planar magnetohydrodynamic compressible
flow with the viscosity depending on the specific volume of the gas and the
heat conductivity proportional to a positive power of the temperature, and  obtained the global
existence  of strong solutions.
However, there are no works about the local
existence of the strong  solution to the Cauchy problem or  initial-boundary value problem for the 3-D compressible MHD equations with degenerate viscosity coefficients and vacuum.
\\
\vspace{2mm}

\noindent{\bf Key Difficulties and Strategies}:\\

In this paper, we focus on the local strong solution for 3-D MHD equations \eqref{prob1} with density-dependent viscosities ( $\delta=1$ in \eqref{T}) in  exterior domains  with Navier-slip boundary conditions. 	
The analysis of the degeneracies in momentum equations and the boundary conditions require some special attentions. The major difficulties include:

$(1)$  The degeneracy of time evolution in momentum equations. For the case $\delta=1$, if  $\rho>0$,  the momentum equation $\eqref{prob1}_2$ can be formally rewritten as
\begin{equation}{\label{u1}}
	u_{t}+u\cdot\nabla u-(\alpha\Delta u+(\alpha+\beta)\nabla\mbox{div}u)+\frac{A\gamma}{\gamma-1}\nabla\rho^{\gamma-1}
	=\frac{\nabla\rho}{\rho}\cdot Q(u)+\frac{1}{\rho}\mbox{curl}H\times H,
\end{equation}
where
\begin{equation}
	Q(u)=\alpha \left(\nabla u+(\nabla u)^T\right)+\beta\mbox{div}u\mathbb{I}_{3}.
\end{equation}
 Inspired by \cite{LPZ20172D} for 2-D shallow water equation, the degeneracies of the time evolution and the viscosity can be transferred to the possible singularity of the special source term $\frac{\nabla\rho}{\rho}\cdot Q(u)$ if the  vacuum only appears at the far field.  And pointed out the possible singularity of the quantities  $\frac{\nabla\rho}{\rho}$ could be well-defined in $D^1\cap D^2$.  Hence for our problem,  one could transfer the degeneration shown in system (\ref{prob1}) caused by the far field
vacuum to the possible singularity of the quantities  $\psi=:\frac{\nabla\rho}{\rho}$ and $J=:\frac{H}{\rho}$. Then provided the special source terms $\frac{\nabla\rho}{\rho}\cdot Q(u)$, and $J\cdot\nabla H-J\cdot(\nabla H)^{T}$ could be dominated well, the velocity $u$ of the fluid can be governed by a strong parabolic system. Besides, $\nabla\rho^{\gamma-1}$, $\gamma>1$ can be rewritten as
\begin{equation}
\nabla\rho^{\gamma-1}=\rho^{\frac{\gamma-1}{2}}\nabla\rho^{\frac{\gamma-1}{2}}.\nonumber
\end{equation}
Because the term $\rho^{\frac{\gamma-1}{2}}$ will play an important role  in analyzing the regularity of $u$,  we introduce the new variable $\phi=:\rho^{\frac{\gamma-1}{2}}$ as in \cite{LPZ20172D}, while $c(x,t)=\sqrt{A\gamma}\rho^{\frac{\gamma-1}{2}}$ is actually the local sound speed.

$(2)$  The boundary conditions for unknown variables. For  the boundary conditions of Navier-slip for the velocity field and the perfect conduction of magnetic filed, it is more suitable to use the $L^2$-norms of $\mbox{div}$ and $\mbox{curl}$ to control that of the derivatives of the velocity and magnetic fields. We should point out that the slip boundary conditions (\ref{boundary2-1}) for $u$ is inhomogeneous, it results in that the boundary terms are not vanished, which should be managed carefully  with trace theorem. Moreover,
a parabolic equation $(\ref{J})_4$ for $J$ with corresponding boundary
\begin{equation}
J\cdot n|_{\partial\Omega}=0,\quad \mbox{curl}J\times n|_{\partial\Omega}=-(\psi\times J)\times n|_{\partial\Omega}=-(\psi\cdot n)J|_{\partial\Omega}\nonumber
\end{equation}
have been found, see Section 3 for   details.  Although they are similar to Navier-slip boundary conditions, but the right-hand side of second condition depends on both $J$ and $\psi$. Noted that such boundary conditions are challenges to deal with,  see (\ref{J1}) and (\ref{3t}) for example. Meanwhile, in order to get estimate $\|\nabla^2 J\|_1$, the elliptic estimate for $J$ has been  established in Appendix \ref{prop-J}.\\

{\bf Main Results} \\

Throughout this paper, we take the following simplified notations for the standard inhomogeneous and homogeneous Sobolev space:
\begin{align*}
	& \|f\|=\|f\|_{L^{2}(\Omega)};\quad 	\|f\|_s=\|f\|_{H^s(\Omega)};\quad D^1=\{f\in L^6(\Omega):\,\,|f|_{D^1}=|\nabla f|_{L^2}<+\infty\};\\
	& D^{k,r}=\{f\in L_{loc}^1(\Omega):\,\,|f|_{D^{k,r}}=|\nabla^kf|_{L^r}<+\infty\},\quad D^k=D^{k,2} \, \,\text{for}\,\,  k\geq 2.
\end{align*}
One can get a detailed and precise study of  homogeneous Sobolev space in \cite{Galdi1994book}.

Based on the above observations, and motivated by \cite{LPZ20172D}, the definition for regular solution to initial-boundary value problem is given as following.

\begin{defn}
(Regular solutions to initial-boundary value problem) Let $T>0$ be a finite constant. $(\rho,u,H)$ is called a regular solution to initial-boundary value problem (\ref{prob1})-(\ref{1.8}) in $\Omega\times[0,T]$ if $(\rho,u,H)$  satisfies\\[2mm]
$(1)$  $(\rho,u,H)$ satisfies the initial-boundary value problem (\ref{prob1})-(\ref{1.8}) a.e. in $(x,t)\in\Omega\times[0,T]$;\\[2mm]
$(2)$  $\rho\geq0$, $\rho^{\frac{\gamma-1}{2}}\in C([0,T_{*}];H^3), \quad (\rho^{\frac{\gamma-1}{2}})_t\in C([0,T_{*}];H^2)$;\\[2mm]
$(3)$  $u\in C([0,T];H^3)\cap L^2([0,T];D^4)$, $u_t\in C([0,T];H^1)\cap L^2([0,T];D^2)$;\\[2mm]
$(4)$ $H\in C([0,T];H^3)\cap L^2([0,T];D^4)$, $H_t\in C([0,T];H^1)\cap L^2([0,T];D^2)$;\\[2mm]
$(5)$   $u_{t}+u\cdot\nabla u-(\alpha\Delta u+(\alpha+\beta)\nabla\mbox{div}u)
=\frac{\nabla\rho}{\rho}\cdot Q(u)+\frac{1}{\rho}\mbox{curl}H\times H$ holds when $\rho(x,t)=0$.
\end{defn}

Now we are ready to state our main results.
\begin{thm}\label{thm1}
If the initial data $(\rho_0,u_0,H_0)$ satisfy the conditions:
\begin{align}\label{initial data}
	\phi_0\equiv\rho_0^{\frac{\gamma-1}{2}}>0,\quad  (\phi_0,u_0,H_0)\in H^{3},\quad \psi_0=\frac{\nabla\rho_0}{\rho_0}
	\in  D^1\cap D^2,\quad \frac{H_0}{\rho_0}\in H^{2}.
\end{align}
Moreover, we assume that the initial data satisfy  the compatibility condition  with the boundary conditions.
Then there exists a time
$T_{*}>0$  and a unique local regular solution $(\rho,u,H)$ to the initial-boundary value problem (\ref{prob1})-(\ref{1.8}). Moreover, if $1<\gamma\leq \frac{5}{3}$ or $\gamma=2, 3$  then $\rho\in C([0,T_{*}];H^3)$.
\end{thm}

\begin{rmk} Theorem \ref{thm1} shows that
	$$\rho^{\frac{\gamma-1}{2}}\in C([0,T_{*}];H^3), \quad \frac{\nabla\rho}{\rho}\in  D^1\cap D^2,$$
it suggests that $\frac{\nabla\rho}{\rho}\in L^\infty$. Hence the vacuum only occurs in the far field.
\end{rmk}

\begin{rmk}
We remark that the initial data \eqref{initial data} contains a large class of functions, for example,
\begin{align}
	\rho_{0}(x)=\frac{1}{1+|x|^{2\sigma}},\quad H_0(x)=\left(\frac{1}{1+|x|^{2\delta_1}},\frac{1}{1+|x|^{2\delta_2}},\frac{1}{1+|x|^{2\delta_3}}\right),\nonumber
\end{align}
where $\sigma>\frac{3}{2(\gamma-1)}$ and $\delta_{i}>\sigma+\frac{3}{4}$, $i=1,2,3$.  The auxiliary function $J$ becomes well-defined provided $H_0$ decays faster than $\rho_0$ to some extent.
\end{rmk}

\begin{rmk} In Theorem \ref{thm1}, the coefficients of viscosity are given as
	\begin{align}\label{vis}
		\mu(\rho)=\alpha \rho,\quad\lambda(\rho)=\beta\rho.
	\end{align}
	In particular, the results obtained in Theorem \ref{thm1} also apply to the local existence of strong solutions for the Cauchy problem to the compressible  MHD equations with degenerate viscosities as \eqref{vis} in $\mathbb{R}^{3}$ which extend the results for  Cauchy problem to compressible Navier-Stokes equations in \cite{LPZ20172D}  to that of the compressible MHD equations in $\mathbb{R}^{3}$.
	
	\end{rmk}
	
\noindent {\bf Structure of the Paper:} \\
\begin{itemize}
\item Section 2 presents preliminary lemmas, including Sobolev inequalities and elliptic estimates.
\item In Section 3 , we first reformulate  the problem in terms of new variables $\phi$, $\psi$ and $J$,  (\ref{J})-(\ref{initial-2}).  Then we establish the a priori estimates on the solutions of the  linearized problem allowing vacuum in the far field. Finally, the existence of local strong solution to the reformulated nonlinear problem (\ref{J})-(\ref{initial-2}) is established by using the iteration with the help of  existence results of the linearized problem  with vacuum.
\item The appendix provides technical elliptic estimates for the auxiliary variable  $J$ with the specific boundary conditions.
 \end{itemize}

Our results not only advance the theoretical understanding of degenerate MHD systems but also lay the groundwork for future studies on global solutions and stability.

\section{Preliminaries}
In this section, we list some basic lemmas that will be used later. The first one is the well-known Gagliardo-Nirenberg inequalities.

\begin{lem}\label{lem-2-1}
Let $\Omega\subset\mathbb{R}^{3}$ be the exterior of a bounded domain with  smooth boundary. Then
for $p\in[2,6]$, $q\in(1,\infty)$, and $r\in(3,\infty)$, there exists some generic constant $C>0$ 
such that for
$f\in H^{1}(\Omega)$, and $g\in L^{q}(\Omega)\cap D^{1,r}(\Omega)$,
it holds that
\begin{align}
\|f\|^p_{L^{p}}\leq C\|f\|^{\frac{6-p}{2}}_{L^{2}}\|\nabla f\|^{\frac{3p-6}{2}}_{L^{2}},\quad
\|g\|_{L^{\infty}}\leq C\|g\|^{\frac{q(r-3)}{3r+q(r-3)}}_{L^{q}}\|\nabla g\|^{\frac{3r}{3r+q(r-3)}}_{L^{r}}.\nonumber
\end{align}
Some special cases of these  inequalities are
\begin{equation}\label{ineq1}
\|u\|_{L^{6}}\leq C\|\nabla u\|_{L^{2}},\quad \|u\|_{L^{\infty}}\leq C\|\nabla u\|_{L^2}^{\frac{1}{2}}\|\nabla^2u\|_{L^2}^{\frac{1}{2}}.\nonumber
\end{equation}
\end{lem}

The second lemma presents  some compactness results obtained via the Aubin-Lions Lemma.
\begin{lem}\label{lem-AL}\cite{A-L}
Let $X_0\subset X\subset X_1$ be three Banach spaces. Suppose that $X_0$ is compactly embedded in $X$, and $X$ is continuously embedded in $X_1$. Then the following statements hold:
\\ $(i)$\  If $F$ is bounded in $L^p([0,T]; X_0)$ for $1\leq p<\infty$, and $F_t$ is bounded in $L^1([0,T]; X)$, then $F$ is relatively compact in $L^p([0,T];X)$;\\
$(ii)$\ If $F$ is bounded in $L^{\infty}([0,T|;X_0)$ and $F_t$ is bounded in $L^{p}([0,T];X_1)$ for $p>1$,
then $F$ is relatively compact in $C([0,T];X)$.
\end{lem}

The next lemma serves as a valuable tool for enhancing weak convergence to the strong convergence.
\begin{lem}\label{lem-strong} \cite{Majda1986book}
If the function sequence $\{w_n\}_{n=1}^{\infty}$converges weakly to $w$ in a Hilbert space $X$, then it converges strongly to $w$ in $X$ if and only if
	\begin{align}
\|w\|_{X}\geq \lim\sup_{n\rightarrow\infty}\|w_n\|_{X}.
\nonumber
\end{align}
\end{lem}

The following lemma  allows one to estimate the $H^{m}$-norm of a vector-valued function $v$ based on  its
$H^{m-1}$-norm of $\mbox{curl}v$ and $\mbox{div} v$ (see \cite{XX2007}).
\begin{lem}\label{lem-div-curl}
Let $\Omega$ be a domain in $\mathbb{R}^{N}$ with smooth boundary $\partial\Omega$ and outward normal $n$. Then there exists a constant $C>0$,
such that
\begin{equation}\label{curl-div}
\|v\|_{H^{s}(\Omega)}\leq C\big(\|\mbox{div}v\|_{H^{s-1}(\Omega)}+\|\mbox{curl}v\|_{H^{s-1}(\Omega)}
+ |v\cdot n|_{H^{s-1/2}(\partial\Omega)}+\|v\|\big),
\end{equation}
and
\begin{equation}
\|v\|_{H^{s}(\Omega)}\leq C\big(\|\mbox{div}v\|_{H^{s-1}(\Omega)}+\|\mbox{curl}v\|_{H^{s-1}(\Omega)}
+ |v\times n|_{H^{s-1/2}(\partial\Omega)}+\|v\|\big),
\end{equation}
for any $v\in H^{s}(\Omega)$, $s\geq1$.
\end{lem}

Now, we present a series of  elliptic estimates for elliptic equations subject to  Navier-slip  boundary conditions in the smooth exterior domain $\Omega$.

\begin{lem}\label{lem-elliptic-1} \cite{DL2024}
Assume that  $\Omega$ is the exterior domain of a bounded  domain in $\mathbb{R}^3$ with smooth boundary.
Suppose that $u$ satisfy the Lam\'{e} equation:
	\begin{equation}\label{equ-u}
		\left\{
		\begin{array}{llll}
			-\mu\Delta u-(\mu+\lambda)\nabla\mbox{div}u=g\quad \mbox{in}\quad \Omega,\\[2mm]
			u\cdot n=0,\quad \mbox{curl} u\times n=-K(x)u\quad \mbox{on}\quad \partial\Omega,
		\end{array}
		\right.
	\end{equation}
where $n$ stands for the outward unit normal to $\partial\Omega$, and  $K(x)$ is a $3\times 3$ symmetric matrix.
There exists a positive constant $C$ depending only on $\mu,\lambda,\Omega$ and the matrix $K$ such that it holds
\begin{equation*}
	\|\nabla^2u\|_{H^{s}(\Omega)}\leq C\big(\|g\|_{H^{s}(\Omega)}+\| \nabla u\|_{L^{2}(\Omega)}\big),
\end{equation*}
where $s=0,1,2$.
\end{lem}

Next, we  give the following estimate about the boundary integrals.

\begin{prop}\label{lem-boun}
Assume that $\Omega$ is the exterior domain of a bounded  domain in $\mathbb{R}^3$
 with smooth boundary. Let $\nabla v\in L^{2}(\Omega)$ with $v\cdot n=0$ on $\partial\Omega$. It
holds for $\nabla f \in L^{2}$, Then for $\nabla f \in L^{2}$ there holds,
\begin{align}
\left|\int_{\partial\Omega}v\cdot\nabla f\right|\leq \|\nabla v\|_{L^2}\|\nabla f\|_{L^2}.
\end{align}
\end{prop}
\begin{proof}
Denote $v^{\bot}=-v\times n$, it follows from the boundary condition $v\cdot n|_{\pl\Omega}=0$ that
\begin{equation*}
v=v^{\bot}\times n\quad \mbox{on}\quad \pl\Omega.
\end{equation*}
Therefore, using some simple identities about vector fields and the divergence theorem, we
conclude that
\begin{align*}
\int_{\partial\Omega}v\cdot\nabla f&=\int_{\partial\Omega}v^{\bot}\times n\cdot\nabla f
=\int_{\partial\Omega}(\nabla f\times v^{\bot})\cdot n\\[2mm]
&=\int_{\Omega}\mbox{div}(\nabla f\times v^{\bot})
=-\int_{\Omega}\nabla f\times \mbox{curl}v^{\bot}\\[2mm]
&\leq \|\nabla f\|_{L^2}\|\nabla v\|_{L^2}.
\end{align*}
This completes the proof.
\end{proof}

At the end of this section, we will provide  some  estimates for $H$, which will be  used later. When $H$ and $v$ smooth sufficiently,  one can see
\begin{equation}\label{fact1}
\mbox{curl}(v\times H)=H\cdot \nabla v-v\cdot\nabla H-H\mbox{div} v.
\end{equation}
due to  $\mbox{div}H = 0$. Moreover, by using the boundary condition $H\cdot n|_{\partial\Omega}=0$ and (\ref{curl-div}), it holds
\begin{equation}\label{fact2}
\|\nabla H\|^2\leq C\left(\|\mbox{curl}H\|^2+\|H\|^2\right).
\end{equation}

\section{Existence of Regular Solutions}
In this section, we will give the proof of Theorem \ref{thm1} by subsections \ref{3.1}-\ref{3.6}.  For this purpose, we first reformulate the original  initial-boundary value problem into a simpler form.
\subsection{Reformulation}\label{3.1}
 In this subsection, introducing some new variables as
 $$\phi=\rho^{\frac{\gamma-1}{2}},\quad  \psi=\frac{\nabla\rho}{\rho}, \quad J=\frac{H}{\rho},$$
  we first reformulate the original initial-boundary value problem (\ref{prob1})-(\ref{1.8}) as follows:

\begin{equation}\label{J}
\left\{
\begin{array}{llll}
\phi_t+u\cdot\nabla\phi+\frac{\gamma-1}{2}\phi\mbox{div}u=0,\\[2mm]
u_{t}+u\cdot\nabla u-(\alpha\Delta u+(\alpha+\beta)\nabla\mbox{div}u)+\frac{2A\gamma}{\gamma-1}\phi\nabla\phi
=\psi\cdot Q(u)+J\cdot\nabla H-J\cdot(\nabla H)^{T},\\[2mm]
H_{t}- \eta \Delta H=\mbox{curl}(u\times H),\quad \quad \mbox{div}H=0,\\[2mm]
J_t+u\cdot\nabla J-\eta\Delta J=J\cdot\nabla u+\eta J\mbox{div}\psi+\eta J\psi^2+2\eta \psi\cdot \nabla J,
\end{array}
\right.
\end{equation}
where
\begin{equation*}
	Q(u)=\alpha\left(\nabla u+(\nabla u)^T\right)+\beta\mbox{div}u I_3,
	\end{equation*}
with the  following boundary conditions:
\begin{equation}\label{boundary2}
 	u\cdot n=0,\quad   \mbox{curl}u\times n=-K(x)u,\quad \mbox{on}\quad \partial\Omega,
 \end{equation}
 \begin{equation}\label{boundary3}
\quad H\cdot n=0,\quad \mbox{curl}H\times n=0, \quad \mbox{on}\quad \partial\Omega,
\end{equation}
and
\begin{equation}\label{bdy-J}
J\cdot n|_{\partial\Omega}=0,\quad \mbox{curl}J\times n|_{\partial\Omega}=-(\psi\times J)\times n|_{\partial\Omega}=-(\psi\cdot n)J|_{\partial\Omega}.
\end{equation}
Because
\begin{align*}
\mbox{curl}J=\frac{1}{\rho}\mbox{curl}H+\nabla\rho\times H=\frac{1}{\rho}\mbox{curl}H-\psi\times J,
\end{align*}
from \eqref{boundary3}, we derive  the boundary conditions for $J$ as follows,
 \begin{align*}
\mbox{curl}J\times n=-\psi\times J\times n=-(\psi\cdot n)J.
\end{align*}
The initial data is given by
 \begin{equation}\label{initial-2}
(\phi,u,H,J)|_{t=0}=(\phi_0,u_0,H_0,J_0)\equiv\left(\rho_0^{\frac{\gamma-1}{2}},u_0,H_0,\frac{H_0}{\rho_0}\right).
\end{equation}
Here we point out that $\psi$ satisfies
	$$\psi_t+\nabla(u\cdot\psi)+\nabla\mbox{div}u=0.$$
	
Now we turn to establish the local well-posedness of the strong solution to (\ref{J})-(\ref{initial-2}). More precisely,
\begin{thm}\label{thm2}
If the initial data $(\phi_0,u_0,H_0,J_0)$ satisfy the following  conditions:
 \begin{equation}\label{intial2}
		\phi_0>0,\quad (\phi_0,u_0,H_0)\in H^3,\quad \psi_0=\frac{2}{\gamma-1}\frac{\nabla\phi_0}{\phi_0}\in D^1\cap D^2,\quad J_0=\frac{H_0}{\rho_0}\in H^2,
	\end{equation}
then there exist a time $T_{*}>0 $ and a unique  local strong solution $(\phi,u,H,J)$ to the initial-boundary value problem (\ref{J})-(\ref{initial-2}) satisfying
\begin{equation}
\begin{split}
&\phi\in C([0,T_{*}];H^3), \quad \phi_t\in C([0,T_{*}];H^2),\\[2mm]
&\psi\in C([0,T_{*}];D^1 \cap D^2)  \quad \psi_t\in C([0,T_{*}];H^1), \quad \psi_{tt}\in L^2(0,T_{*};L^2)\\[2mm]
&H\in C([0,T_{*}];H^3)\cap L^2(0,T_{*}; H^4), \\[2mm]
 & H_t\in C([0,T_{*}];H^1)\cap L^2(0,T_{*}; H^2), \quad H_{tt}\in L^2(0,T_{*};L^2),
\\[2mm]
 &J\in L^{\infty}([0,T_{*}];H^2)\cap L^2(0,T_{*}; H^3),
  J_t\in L^{\infty}([0,T_{*}];L^2)\cap L^2(0,T_{*}; H^1),
\\[2mm]
 &u\in C([0,T_{*}];H^3)\cap L^2(0,T_{*}; H^4), \\[2mm]
 & u_t\in C([0,T_{*}];H^1)\cap L^2(0,T_{*}; H^2), \quad u_{tt}\in L^2(0,T_{*};L^2). \nonumber
 \end{split}
 \end{equation}
\end{thm}

The proof of Theorem \ref{thm2} will be divided into four subsequent subsections, \ref{3.2}-\ref{3.5}.
In Subsection \ref{3.6}, it will be demonstrated that Theorem \ref{thm2} implies Theorem \ref{thm1}.

\subsection{Linearization}\label{3.2}

In order to proceed with nonlinear system (\ref{J}), we first need to address  the following linearized problem:
\begin{equation}\label{linear}
\left\{
\begin{array}{llll}
\phi_t+v\cdot\nabla\phi+\frac{\gamma-1}{2}\phi\mbox{div}v=0,\\[2mm]
u_{t}-\left(\alpha\Delta u+(\alpha+\beta)\nabla\mbox{div}u\right)
=-v\cdot\nabla v-\frac{2A\gamma}{\gamma-1}\phi\nabla\phi+\psi\cdot Q(v)+J\cdot\nabla H-J\cdot (\nabla H)^{T},\\[2mm]
H_{t}- \eta \Delta H=\mbox{curl}(v\times H),\\[2mm]
J_t-\eta\Delta J=-v\cdot\nabla J+J\cdot\nabla v+\eta J\mbox{div}\psi+\eta J\psi^2+2\eta \psi\cdot \nabla J,
\end{array}
\right.
\end{equation}
where $v$ is a  known vector satisfying
\begin{equation}
v(x,0)=u_0(x),
\end{equation}
and $v$ has the same boundary conditions as $u$, that is
\begin{equation}
	\begin{split}
		&v\cdot n|_{\partial\Omega}=0,\quad  \mbox{curl}v\times n|_{\partial\Omega}=- K(x)v.
		\end{split}
	\end{equation}
Moreover,  $v$ belongs  to the following spaces:
\begin{equation}\label{reg1}
v\in C([0,T]; H^3)\cap  L^2([0,T]; H^4),\quad v_t\in C([0,T]; H^1)\cap  L^2([0,T]; H^2).
\end{equation}
The unknown variables $(u,H,J)$ meet the following boundary conditions:
\begin{equation}
 	u\cdot n=0,\quad  \mbox{curl}u\times n=-K(x)u\quad \mbox{on}\quad \partial\Omega,
 \end{equation}
 \begin{equation}
\quad H\cdot n=0,\quad \mbox{curl}H\times n=0 \quad \mbox{on}\quad \partial\Omega,
\end{equation}
and
\begin{equation}\label{bdy-Jl}
J\cdot n=0,\quad \mbox{curl}J\times n=-(\psi\times J)\times n=-(\psi\cdot n)J\quad \mbox{on}\quad \partial\Omega.
\end{equation}
We assume that
\begin{align}
\phi_0>0, \quad (\phi_0-\phi^{\infty},u_0, H_0)\in H^{3},\quad  J_0\in H^{2},
\end{align}
where $\phi^{\infty}\geq 0$ is a constant. At the same time, $\psi=\frac{2}{\gamma-1}\frac{\nabla\phi}{\phi}$ satisfies the following equation
\begin{equation*}
	\psi_t+\nabla(v\cdot\psi)+\nabla\mbox{div}v=0.
	\end{equation*}
A straightforward  calculation indicates $\partial_i\psi_j=\partial_j\psi_i$, which implies the above equation is  equivalent to
\begin{equation}\label{linear2}
	\psi_t+\sum_{l=1}^3A_l\partial_l\psi+B\psi+\nabla\mbox{div}v=0,
	\end{equation}
where
\begin{equation*}
	A_l=\begin{bmatrix}
		v_l&0&0\\
		0&v_l&0\\
		0&0&v_l
		\end{bmatrix}, \quad  l=1,2,3,
	\end{equation*}
 $B=(\nabla v)^{T}$. So equation (\ref{linear2}) is a positive symmetric hyperbolic system.

In the subsequent two subsections, we first solve the linearized problem \eqref{linear}-(\ref{bdy-Jl}) when the initial density is away
from vacuum ($\inf\phi_0>0$), then we will establish some  uniform estimates which are independent of the lower bound of $\phi_0$.

The following global well-posedness of a strong solution to (\ref{linear})-(\ref{bdy-Jl}) can be obtained  when $\phi_0\geq \sigma$ for some positive constant $\sigma$.
\begin{lem}\label{global}
Assume that $(\phi_0,u_0,H_0, J_0=\frac{H_0}{\rho_0})$ satisfies
\begin{equation*}
\phi_0\geq\sigma, \quad \phi_0-\phi^{\infty}\in H^3,\quad (u_0,H_0)\in H^{3},\quad J_0\in H^2,
\quad \psi_0=\frac{2}{\gamma-1}\frac{\nabla\phi_0}{\phi_0}\in D^1\cap D^2,
\end{equation*}
for some positive constant $\sigma$. Then for any finite constant $T>0$, there exists a unique strong solution $(\phi, u, H, J)$ in $[0,T]\times \Omega$
to (\ref{linear})-(\ref{bdy-Jl}) such that
\begin{align*}
&\phi\geq \underline{\sigma},\quad \phi-\phi^{\infty}\in C([0,T]; H^3),\quad \phi_t\in C([0,T]; H^2)\nonumber\\[2mm]
&\psi\in C([0,T]; D^1\cap D^2),\quad \psi_t\in C([0,T]; H^1),\quad  \psi_{tt}\in L^2(0,T; L^2)\nonumber\\[2mm]
&  u\in C([0,T]; H^3)\cap L^2([0,T]; H^4), \quad  u_t\in C([0,T]; H^1)\cap L^2([0,T]; H^2),\quad u_{tt}\in L^2(0,T; L^2)\nonumber\\[2mm]
&H\in C([0,T]; H^3)\cap L^2([0,T]; H^4),\quad H_t\in C([0,T]; H^1)\cap L^2([0,T]; H^2)\nonumber\\[2mm]
&J\in L^{\infty}([0,T];H^2)\cap L^2(0,T; H^3), \quad J_{t}\in L^{\infty}([0,T];L^2)\cap L^2(0,T; H^1),
\end{align*}
where  $\underline{\sigma}$ is a  positive constant depending only on $\sigma$ and $T$.
\end{lem}
\begin{proof}
The existence, regularity and positivity of a unique solution $\phi$ to $(\ref{linear})_1$ can be obtained
essentially via the standard theory of transport equation. Next, we present a priori estimates for $\psi$ based on the standard energy estimate theory for positive symmetric hyperbolic systems. Utilizing the regularity of $\phi$ and $\psi$, along with Lemmas \ref{lem-div-curl}-\ref{lem-elliptic-1} and Appendix \ref{prop-J}, it is easy to determine $(J,H,u)$ from the linear parabolic equations. Consequently, we omit the details.
\end{proof}

\subsection{A priori estimates independent of $\sigma$}\label{3.3}
Now we are going to establish some uniform estimates on the solutions derived in Lemma \ref{global}. To this
end, we create a fixed value $T>0$ and a positive constant $c_0$ large enough such that
\begin{equation}\label{initial-0}
		2+\phi^{\infty}+\|\phi_0-\phi^{\infty}\|_3+\|u_0\|_3+\|H_0\|_3+\|J_0\|_2+|\psi_0 |_{D^1\cap D^2}\leq c_0.
	\end{equation}
Assume that there exist some time $T^{*}\in(0,T)$ and constants $c_i$ ($i = 1, 2, 3$) such that
\begin{equation*}
1\leq c_0\leq c_1\leq c_2\leq c_3,
\end{equation*}
and
\begin{align}\label{known}
&\sup_{0\leq t\leq T^{*}}\|v(t)\|^2_1+\int_0^{T^{*}}\left(\|\nabla v\|^2_1+\|v_t\|^2\right)dt\leq c_1^2,\nonumber\\[2mm]
&\sup_{0\leq t\leq T^{*}}\left(\|\nabla^2v(t)\|^2+\|v_t(t)\|^2\right)+\int_0^{T^{*}}\left(\|\nabla^3v\|^2+\|\nabla v_t\|^2\right)dt\leq c_2^2,\nonumber\\[2mm]
&\sup_{0\leq t\leq T^{*}}\left(\|\nabla^3v(t)\|^2+\|\nabla v_t(t)\|^2\right)+\int_0^{T^{*}}\left(\|\nabla^4v\|^2+\|\nabla^2v_t\|^2+\|v_{tt}\|^2\right)dt\leq c_3^2.
\end{align}

We will confirm $T^{*}$ and $c_i$ ($i=1,2,3$) subsequently, ensuring they depend only on $c_0$, $T$ and the fixed constants $\alpha,\beta,\gamma$ and $\eta$, as referenced in (\ref{constant}).

Let $(\phi, H, J, u)$ be the unique strong solution to (\ref{linear})-(\ref{bdy-Jl}) on $[0,T]\times \Omega$. A series of uniform local (in time) estimates will be established below.  Hereinafter, we use $C\geq 1$ to denote a generic positive constant that depends only on fixed constants $\alpha,\beta,\gamma$, $\eta$ and $T$.

It is observed that the equations for $\phi$ and $\psi$ are independent of $H$ and $J$, hence the a priori estimates for $\phi$ in Lemma \ref{lem-phi} and for $\psi$ in Lemma \ref{lem-psi} seem like those for the Cauchy problem of the Navier-Stokes equation (see Lemmas 3.2-3.3 in \cite{LPZ20172D}). We provide the proof here for the sake of completeness. We will now proceed with the estimations for $\phi$.
\begin{lem}\label{lem-phi}
Let $(\phi, H, J, u)$ be the unique  strong solution to (\ref{linear})-(\ref{bdy-Jl}) on $[0,T]\times \Omega$. Then for $0\leq t\leq T_1=\min\{T^{*},(1+c_3)^{-2}\}$, it holds
\begin{equation*}\label{phi-lem}
	\begin{split}
&\|\phi(t)\|_{\infty}+\|\phi(t)-\phi^{\infty}\|_3\leq Cc_0,\quad
\|\phi_t(t)\|\leq Cc_0c_1,\\[2mm]
&\quad \|\phi_t(t)\|_{D^1}\leq Cc_0c_2,\quad\|\phi_t(t)\|_{D^2}\leq Cc_0c_3.
\end{split}
\end{equation*}
\end{lem}
\begin{proof}
Since
\begin{equation*}
	\int_0^t\|v\|_4ds\leq t^{\frac{1}{2}}\left(	\int_0^t\|v\|^2_4ds\right)^{\frac{1}{2}}\leq c_3t^{\frac{1}{2}}.
	\end{equation*}
Then,  the stand energy estimates for transport equations imply that
\begin{equation}
	\begin{split}
\|\phi(t)-\phi^{\infty}\|_3
&\leq \left(\|\phi_0-\phi^{\infty}\|_3+\phi^{\infty}\int_0^t\|\nabla v\|_3ds\right)\exp\{C\int_0^t\|v\|_4ds\}\nonumber\\[2mm]
&\leq Cc_0,
\end{split}
\end{equation}
for $0\leq t\leq T_1=\min\{T^{*},(1+c_3)^{-2}\}$. Moreover, from the equation $(\ref{linear})_1$, one has
\begin{align*}
&\|\phi_t\|^2\leq C\left(\|v\|^2_{L^6}\|\nabla\phi\|^2_{L^3}+\|\phi\|^2_{L^{\infty}}\|\mbox{div}v\|^2\right)\leq Cc^2_0c^2_1,\\[2mm]
&\|\phi_t\|^2_{D^1}\leq C\left(\|v\|^2_{L^{\infty}}\|\phi\|^2_{D^2}+\|\nabla v\|^2_{L^6}\|\nabla\phi\|^2_{L^3}+\|\phi\|^2_{L^{\infty}}\|v\|^2_{D^2}\right)
 \leq C c^2_0c^2_2,
\end{align*}
and
\begin{equation*}
\begin{split}
\|\phi_t\|^2_{D^2}&\leq C\left(\|v\|^2_{L^{\infty}}\|\phi\|^2_{D^3}+\|\nabla \phi\|^2_{L^3}\|\nabla^2 v\|^2_{L^6}
+\|\nabla^2 \phi\|^2_{L^3}\|\nabla v\|^2_{L^6}+\|\phi\|^2_{L^{\infty}}\|v\|^2_{D^3}\right)\nonumber\\[2mm]
&\leq C c^2_0c^2_3,
\end{split}
\end{equation*}
for $0\leq t\leq T_1$.
This completes the proof.
\end{proof}

Now we estimate $\psi$ by the stand energy estimates for positive symmetric hyperbolic
system, which will be employed to address the estimates for $J$ and $u$.
\begin{lem}\label{lem-psi}
Let $(\phi, H, J, u)$ be the unique strong solution to (\ref{linear})-(\ref{bdy-Jl}) on $[0,T]\times \Omega$. Then for $0\leq t\leq T_1$, it holds
\begin{equation*}\label{psi}
\|\psi(t)\|_{\infty}+\|\nabla \psi(t)\|_{1}\leq C c_0,\quad
\|\psi_t(t)\|\leq Cc_0c_2,\quad \|\nabla\psi_t(t)\|^2\leq Cc_0^2c^2_3.
\end{equation*}
\end{lem}
\begin{proof}
Let $\xi=(\xi_1,\xi_2,\xi_3)^{T} (\xi_i=0,1,2) $ and $1\leq|\xi|\leq 2$.
Taking $\partial^{\xi}_{x}$ to equation (\ref{linear2}), one has
\begin{align*}
&\  (\partial^{\xi}_{x}\psi)_t+\sum_{k=1}^3A_k\partial_k\partial^{\xi}_{x}\psi+B\partial^{\xi}_{x}\psi+\partial^{\xi}_{x}\nabla\mbox{div}v\nonumber
\\[2mm]
&=\sum_{k=1}^3\left(A_k\partial_k\partial^{\xi}_{x}\psi-\partial^{\xi}_{x}(A_k\partial_k\psi)\right)
+\left(B\partial^{\xi}_{x}\psi-\partial^{\xi}_{x}(B\psi)\right)
\equiv R_1+R_2.
\end{align*}
Multiplying  the above equation  by  $\partial^{\xi}_{x}\psi$, and integrating by parts  with the boundary condition $v\cdot n|_\Omega=0$, we can infer that
\begin{align}\label{psi-1}
\frac{1}{2}\frac{d}{dt}\|\partial^{\xi}_{x}\psi\|^2\leq& \left(\sum_{k=1}^3\|\partial_{k}A_k\|_{L^{\infty}}+\|B\|_{L^{\infty}}\right)\|\partial^{\xi}_{x}\psi\|^2\nonumber\\[2mm]
&+\left(\|R_1\|+\|R_2\|+\|\partial^{\xi}_{x}\nabla^2v\|\right)\|\partial^{\xi}_{x}\psi\|.
\end{align}
When $|\xi|=1$, it is clear to have
\begin{align*}
\|R_1\|\leq
\|\partial_{x}A_k\partial_{k}\psi\|
\leq C\|\nabla v\|_{L^{\infty}} \|\nabla\psi\|,
\end{align*}
and
\begin{align*}
\|R_2\|\leq\|\partial_{x}B\psi\|
\leq C\|\nabla^2 v\|_{L^3}\|\psi\|_{L^{6}}.
\end{align*}
Putting the estimates of $\|R_1\|$ and  $\|R_2\|$ into (\ref{psi-1}), one obtains
\begin{align*}
\frac{1}{2}\frac{d}{dt}\|\nabla\psi\|^2\leq C \left(\|\nabla v\|_{L^{\infty}}+\|\nabla^2v\|_{L^3}\right)\|\nabla\psi\|^2
+\|\nabla ^3v\|  \|\nabla\psi\|,
\end{align*}
that is
\begin{align*}
\frac{1}{2}\frac{d}{dt}\|\nabla\psi\|\leq C \left(\|\nabla v\|_{L^{\infty}}+\|\nabla^2v\|_{L^3}\right)\|\nabla\psi\|
+\|\nabla ^3v\|.
\end{align*}
Then, by using the Gronwall's inequality,  for $0\leq t\leq T_1=\min\{T^{*},(1+c_3)^{-2}\}$, it implies
\begin{align}\label{psi-2}
\|\nabla\psi\|\leq \left(\|\nabla \psi_0\|+\int_{0}^t\|\nabla^3v\|ds\right)\exp(C\int_0^t\|\nabla v\|_2 ds)\leq Cc_0.
\end{align}
When $|\xi|=2$, we have
\begin{align*}
\|R_1\|&\leq
\|\partial_{x}A_k\partial_x\partial_{k}\psi+\partial^{2}_{x}A_k\partial_k\psi\|
\leq C\|\nabla v\|_{L^{\infty}}\|\nabla^2\psi\|+\|\nabla^2v\|_{L^3}\|\nabla\psi\|_{L^{6}}\\[2mm]
&\leq C\|\nabla v\|_2 \|\nabla^2\psi\|,
\end{align*}
and
\begin{align*}
\|R_2\|&\leq\|\partial^2_{x}B\psi+\partial_{x}B\partial_{x}\psi\|
\leq C\|\nabla^3 v\|\|\psi\|_{L^{\infty}}+\|\nabla^2v\|_{L^3}\|\nabla\psi\|_{L^{6}}\\[2mm]
&\leq C \|\nabla^3 v\|\|\nabla\psi\|^{1/2}\|\nabla^2\psi\|^{1/2}+\|\nabla^2v\|_{L^3}\|\nabla^2\psi\|.
\end{align*}
Adding  the above estimates  into (\ref{psi-1}) yields
\begin{align*}
\frac{1}{2}\frac{d}{dt}\|\nabla^2\psi\|^2\leq C \|\nabla v\|_3\|\nabla^2\psi\|^2+ \|\nabla ^3v\|\|\nabla\psi\|^{1/2}\|\nabla^2\psi\|^{3/2}
+\|\nabla ^4v\|  \|\nabla^2\psi\|,
\end{align*}
which, combined with Cauchy's inequality, suggests
\begin{align*}
\frac{1}{2}\frac{d}{dt}\|\nabla^2\psi\|\leq C \|\nabla v\|_3\|\nabla^2\psi\|+\|\nabla ^3v\|\|\nabla\psi\|+\|\nabla ^4v\|.
\end{align*}
Subsequently, employing the Gronwall's inequality and (\ref{psi-2}),  for $0\leq t\leq T_1=\min\{T^{*},(1+c_3)^{-2}\}$, the following holds
\begin{align}\label{psi-3}
\|\nabla^2\psi\|\leq \left(\|\nabla^2 \psi_0\|+\int_{0}^t\left(\|\nabla^3v\|\|\nabla\psi\|+\|\nabla^4 v\|\right)ds\right)\exp(C\int_0^t\|\nabla v\|_3 ds)\leq Cc_0.
\end{align}
Consequently, the application of the Gagliardo-Nirenberg inequality indicates that (\ref{psi-2}) and (\ref{psi-3}) imply
\begin{equation*}
\|\psi\|_{L^{\infty}}\leq Cc_0.
\end{equation*}
Moreover, the equation
\begin{equation*}
	\psi_t=-\nabla(v\cdot\psi)-\nabla\mbox{div}v,
	\end{equation*}
 yields
\begin{align*}
\|\psi_t\|\leq C\left(\|v\|_{L^{\infty}}\|\nabla\psi\|+\|\psi\|_{L^{\infty}}\|\nabla v\|+\|\nabla\mbox{div}v\|\right)\leq Cc_0c_2,
\end{align*}
\begin{align*}
\|\nabla\psi_t\|\leq C\left(\|v\|_{L^{\infty}}\|\nabla^2\psi\|+\|\psi\|_{L^{\infty}}\|\nabla^2 v\|
+\|\nabla v\|_{L^3}\|\nabla\psi\|_{L^6}+\|\nabla^3v\|\right)\leq Cc_0c_3,
\end{align*}
 for $0\leq t\leq T_1$. The proof  is completed.
\end{proof}

Now, we are going to get the estimates for $J$,
which are crucial for deriving the estimates of $u$ in Lemma \ref{lem-u}.  Since  that boundary conditions (\ref{bdy-Jl}) of $J$ contains itself and $\psi$,
it is imperative to carefully analyze the boundary term to establish the elliptic estimation of $J$. To facilitate the reader, the proof of the elliptic estimation of $J$ will be provided in the appendix (Appendix \ref{prop-J}).

\begin{lem}\label{lem-J}
Let $(\phi,u, H, J)$ be the unique strong solution to (\ref{linear})-(\ref{bdy-Jl}) on $[0,T]\times \Omega$. Then for $0\leq t\leq T_2=\min\{T^{*},(1+c_{3})^{-6}\}$, it holds
\begin{equation}
\|J\|^2+\int_{0}^{t}\|\nabla J\|^2ds\leq C c^2_0,\nonumber
\end{equation}
\begin{equation}
\|\nabla J\|^2+\int_0^t(\|J_t\|^2+\|\nabla^2J\|^2)ds\leq C c_0^4,\nonumber
\end{equation}
\begin{equation}
	\|J_t\|^2+\int_0^t(\|\nabla J_t\|^2+\|\nabla^{3}J\|^2)ds\leq C c_0^6,\quad	\|\nabla^2J\|^2\leq C c_1c_2c_0^6.\nonumber
\end{equation}
\end{lem}
\begin{proof}
We recall the linear equation $(\ref{linear})_4$ of $J$ and identify the right side as $F$, namely,
\begin{align}\label{equ-J}
J_t-\eta\Delta J=-v\cdot\nabla J+J\cdot\nabla v+\eta J\mbox{div}\psi+\eta J\psi^2+2\eta \psi\cdot \nabla J\equiv F.
\end{align}
The proof of this lemma will be finished by three steps.

{\it Step 1:} We aim to find $\|J\|$ firstly. By multiplying the equation (\ref{equ-J}) by $J$, and integrating by parts, one obtains
\begin{align}\label{2}
\frac{1}{2}\frac{d}{dt}\|J\|^2+\eta\|\mbox{curl}J\|^2+\eta\|\mbox{div}J\|^2=\eta\int_{\partial\Omega}(\mbox{curl}J\times n)\cdot JdS+\int_{\Omega}F\cdot Jdx,
\end{align}
where we have used the following identity
\begin{align*}
-\eta\int_{\Omega}\Delta J\cdot Jdx&=\eta\int_{\Omega}\mbox{curl}^2 J\cdot Jdx-\eta\int_{\Omega}\nabla
\mbox{div} J\cdot Jdx\\[2mm]
&=\eta\int_{\Omega}|\mbox{curl}J|^2dx+\eta\int_{\Omega}|\mbox{div}J|^2dx-\eta\int_{\partial\Omega}(\mbox{curl}J\times n)\cdot JdS,
\end{align*}
thanks to $J\cdot n|_{\partial\Omega}=0$.

 By using the boundary condition (\ref{bdy-Jl}), the boundary term on \eqref{2} could be handled as follows,
\begin{align}\label{3}
\int_{\partial\Omega}(\mbox{curl}J\times n)\cdot JdS&=-\int_{\partial\Omega}(\psi \cdot n)|J|^2dS\nonumber\\[2mm]
&=-\int_{\Omega}\mbox{div}\left(\psi|J|^2\right)dx\nonumber\\[2mm]
&=-\int_{\Omega}|J|^2\mbox{div}\psi dx-2\int_{\Omega}\psi\cdot\nabla J\cdot J dx\nonumber\\[2mm]
&\leq\|J\|_{L^6}\|J\|_{L^2}\|\mbox{div}\psi\|_{L^3}+\|\nabla J\|_{L^2}\|J\|_{L^2}\|\psi\|_{L^{\infty}}\nonumber\\[2mm]
&\leq\varepsilon\|\nabla J\|^2+C\|J\|^2\left(\|\mbox{div}\psi\|^2_{L^3}+\|\psi\|^2_{L^{\infty}}\right).
\end{align}
We reformulate that last term on the right-hand side of \eqref{2} as
\begin{align}
\int_{\Omega}F\cdot J dx= \int_{\Omega}\left( -v\cdot\nabla J+J\cdot\nabla v+\eta J\mbox{div}\psi+\eta J\psi^2+2\eta \psi\cdot \nabla J\right)\cdot J dx\equiv \sum_{i=1}^5I_i,
\end{align}
and  estimate these terms individually as follows:
\begin{align}\label{I1}
&I_1\equiv-\int_{\Omega}v\cdot\nabla J\cdot Jdx\leq \|v\|_{L^{\infty}}\|\nabla J\|\|J\|\leq \varepsilon \|\nabla J\|^2+C(\varepsilon)\|v\|^2_{L^{\infty}}\|J\|^2,\\[2mm]
&I_2\equiv\int_{\Omega}J\cdot\nabla v\cdot J dx\leq \|J\|_{L^6}\|\nabla v\|_{L^3}\|J\|
\leq \varepsilon \|\nabla J\|^2+C(\varepsilon)\|\nabla v\|^2_{L^3}\|J\|^2,\\[2mm]
&I_3\equiv\int_{\Omega}\eta J\mbox{div}\psi\cdot J dx\leq C \|J\|_{L^6}\|\mbox{div}\psi\|_{L^3}\|J\|
\leq \varepsilon \|\nabla J\|^2+C(\varepsilon)\|\mbox{div}\psi\|^2_{L^3}\|J\|^2,\\[2mm]
&I_4\equiv\int_{\Omega}\eta\psi^2 J\cdot J dx\leq C \|\psi\|^2_{L^{\infty}}\|J\|^2,\\[2mm]
&I_5\equiv\int_{\Omega}2\eta \psi\cdot \nabla J\cdot J dx \leq C \|\psi\|_{L^{\infty}} \|\nabla J\|\|J\|
\leq \varepsilon \|\nabla J\|^2+C(\varepsilon)\|\psi\|^2_{L^{\infty}}\|J\|^2.
\end{align}
Consequently, by substituting above into (\ref{2}), employing (\ref{curl-div}) for $J$, and selecting  $\varepsilon$ small, one derives
\begin{equation*}
		\frac{d}{dt}\|J\|^2+\|\nabla J\|^2	\leq C \left(\|\mbox{div}\psi\|^2_{L^3}+\|\psi\|^2_{L^{\infty}}+\|v\|^2_{L^{\infty}}+\|\nabla v\|^2_{L^3}+1\right) \|J\|^2,
\end{equation*}
 which, together with Gronwall's inequality, yields for $0\leq t\leq T_1= \min\{T^{*},(1+c_{3})^{-2}\}$, that
\begin{align}\label{p-H-2}
	\|J(t)\|^2+\int_0^t\|\nabla J(s)\|^2ds\leq& C\|J_0\|^2
	 \exp\left\{C\int_{0}^{t}\left(\|\mbox{div}\psi\|^2_{L^3}+\|\psi\|^2_{L^{\infty}}+\|v\|^2_{L^{\infty}}+\|\nabla v\|^2_{L^3}+1\right)ds\right\}\nonumber\\[2mm]
	\leq &C c_0^2,
	\end{align}
where Lemma \ref{lem-psi} has been utilized.

{\it Step 2:} We now  estimate $\|\nabla J\|$. By multiplying the equation (\ref{equ-J}) by $J_t$, integrating it over $\Omega$, and observing that
\begin{align*}
-\eta\int_{\Omega}\Delta J\cdot J_tdx&=\eta\int_{\Omega}\mbox{curl}^2 J\cdot J_tdx-\eta\int_{\Omega}\nabla
\mbox{div} J\cdot J_tdx\\[2mm]
&=\frac{\eta}{2}\frac{d}{dt}\left(\int_{\Omega}|\mbox{curl}J|^2dx+\eta\int_{\Omega}|\mbox{div}J|^2dx\right)-\eta\int_{\partial\Omega}(\mbox{curl}J\times n)\cdot J_tdS,
\end{align*}
then the  following holds
\begin{equation}\label{nablaJ}
\frac{\eta}{2}\frac{d}{dt}\left(\|\mbox{curl}J\|^2+\|\mbox{div}J\|^2\right)+\|J_t\|^2-\eta\int_{\partial\Omega}(\mbox{curl}J\times n)\cdot J_{t}dS=\int_{\Omega}F\cdot J_{t}dx.
\end{equation}
By using the boundary condition (\ref{bdy-Jl}) again, we have
\begin{align}\label{J1}
&\int_{\partial\Omega}(\mbox{curl}J\times n)\cdot J_tdS=-\int_{\partial\Omega}(\psi \cdot n)J \cdot J_tdS\nonumber\\[2mm]
&=-\int_{\Omega}\mbox{div}\left(\psi (J \cdot J_t)\right)dx\nonumber\\[2mm]
&=-\int_{\Omega}\mbox{div}\psi J \cdot J_t dx-\int_{\Omega}\psi\cdot \nabla J \cdot J_t dx
-\int_{\Omega}\psi \cdot\nabla J_t\cdot J dx\nonumber\\[2mm]
&=-\int_{\Omega}\mbox{div}\psi J \cdot J_tdx
-\frac{d}{dt}\int_{\Omega}\psi \cdot\nabla J\cdot Jdx+\int_{\Omega}\psi_t \cdot\nabla J\cdot J dx \nonumber\\[2mm]
&\leq -\frac{d}{dt}\int_{\Omega}\psi \cdot\nabla J\cdot Jdx +\varepsilon\|J_t\|^2+C\|\mbox{div}\psi\|^2_{L^3}\| J\|^2_{L^6}+\|\psi_t\|_{L^3}\|\nabla J\|\|J\|_{L^6}\nonumber\\[2mm]
&\leq -\frac{d}{dt}\int_{\Omega}\psi \cdot\nabla J\cdot Jdx +\varepsilon\|J_t\|^2+C\|\mbox{div}\psi\|^2_{L^3}\| \nabla J\|^2+\|\psi_t\|_{L^3}\|\nabla J\|^2.
\end{align}
Utilizing Cauchy's inequality, it is readily apparent
\begin{align}\label{J2}
\int_{\Omega} F\cdot J_tdx&\leq \frac{1}{4} \|J_t\|^2+C\|F\|^2\nonumber\\[2mm]
&\leq \frac{1}{4} \|J_t\|^2+C\left(\|v\cdot\nabla J\|^2+\|J\cdot\nabla v\|^2+\|J\mbox{div}\psi\|^2+\| J\psi^2\|^2+\| \psi\cdot \nabla J\|^2\right)\nonumber\\[2mm]
&\leq \frac{1}{4} \|J_t\|^2+C\left(\|v\|^2_{L^{\infty}}+\|\nabla v\|_{L^3}^2+\|\mbox{div}\psi\|_{L^3}^2+\|\psi\|_{L^6}^4+\| \psi\|^2_{L^{\infty}} \right)\|\nabla J\|^2.
\end{align}
Putting (\ref{J1}) and (\ref{J2}) into (\ref{nablaJ}), choosing $\varepsilon$ small enough, it yields that
\begin{align}\label{4}
&\eta\frac{d}{dt}\left(\|\mbox{curl}J\|^2+\|\mbox{div}J\|^2\right)+2\eta\frac{d}{dt}\int_{\Omega}\psi \cdot\nabla J\cdot Jdx+\|J_t\|^2\nonumber\\[2mm]
&\leq C\left(\|v\|^2_{L^{\infty}}+\|\nabla v\|_{L^3}^2+\|\mbox{div}\psi\|_{L^3}^2+\|\psi\|_{L^6}^4+\| \psi\|^2_{L^{\infty}} +\|\psi_t\|_{L^3} \right)\|\nabla J\|^2.
 \end{align}
Integrating (\ref{4}) over $s\in(0,t)$, $0\leq t\leq \widetilde{T_2}= \min\{T^{*},(1+c_{3})^{-4}\}$, and using  Young's inequality,  invoking Lemma \ref{lem-psi},  it implies
\begin{align}
\|&\mbox{curl}J(t)\|^2+\|\mbox{div}J(t)\|^2+\int_0^t\|J_t(s)\|^2ds\nonumber\\[2mm]
\leq& C\|\nabla J_0\|^2-2\eta\int_{\Omega}(\psi \cdot\nabla J\cdot J)(t)dx +2\eta\int_{\Omega}\psi_0 \cdot\nabla J_0\cdot J_0dx\nonumber\\[2mm]
&+C(c_1c_2+c_0^2+c_0^4 +c_0c_3)\int_0^t\|\nabla J(s)\|^2ds\nonumber\\[2mm]
\leq& C\|\nabla J_0\|^2+\varepsilon\|\nabla J\|^2+C(\varepsilon)\|\psi\|^2_{L^{\infty}}\|J\|^2
+C(c_1c_2+c_0^2+c_0^4 +c_0c_3)\int_0^t\|\nabla J(s)\|^2ds.\nonumber
\end{align}
Subsequently, by using (\ref{curl-div}) and (\ref{p-H-2}) and selecting  $\varepsilon$ small enough, it is demonstrated
\begin{align}
&\|\nabla J(t)\|^2+\int_0^t\|J_t(s)\|^2ds\nonumber\\[2mm]
 &\leq C\|J\|^2+ C\|\nabla J_0\|^2+C\|\psi\|^2_{L^{\infty}}\|J\|^2
+C(c_1c_2+c_0^2+c_0^4 +c_0c_3)\int_0^t\|\nabla J(s)\|^2ds\nonumber\\[2mm]
&\leq  Cc_0^4+C(c_1c_2+c_0^2+c_0^4 +c_0c_3)\int_0^t\|\nabla J(s)\|^2ds,\nonumber
\end{align}
which, along with the Gronwall's inequality, indicates immediately that
\begin{equation}\label{5}
\|\nabla J(t)\|^2+\int_0^t\|J_t(s)\|^2ds
\leq Cc_0^4,
\end{equation}
for  $0\leq t\leq \widetilde{T_2}= \min\{T^{*},(1+c_{3})^{-4}\}$.

Moreover, applying the elliptic estimate \eqref{2J} in Appendix \ref{prop-J} to $J$,  for $0\leq t\leq \widetilde{T_2}$, one can show that
\begin{align}
\int_0^t\|\nabla^2J(s)\|^2ds
&\leq C\int_0^t \left(\|J_t\|^2+\|F\|^2+\| J\|_1^2+\|\psi\|^2_{L^{\infty}} \|J\|^2_1+\|\nabla\psi\|^2_{L^3}\|\nabla J\|^2\right)ds\nonumber\\[2mm]
&\leq Cc_0^4.\nonumber
\end{align}

{\it Step 3:} \ We would like to get the estimate of $\|\nabla^2J\|$. Differentiating $(\ref{equ-J})$ with respect to $t$, multiplying the resulting identity by $J_t$, and after integrating by parts, we obtain
\begin{align}\label{2t}
\frac{1}{2}\frac{d}{dt}\|J_t\|^2+\eta\|\mbox{curl}J_t\|^2+\eta\|\mbox{div}J_t\|^2-\eta\int_{\partial\Omega}(\mbox{curl}J_t\times n)\cdot J_tdS=\int_{\Omega}F_t\cdot J_tdx.
\end{align}
Recall  the definition of $J$ and the boundary condition (\ref{bdy-Jl}) of $J$ to have
\begin{equation}\label{bdy-Jt}
\mbox{curl}J_t\times n =-((\psi\cdot n)J)_t=- \psi\cdot n J_t - \psi_t\cdot n J ,  \hspace{4mm} \mbox{on} \hspace{3mm} \partial\Omega.
\end{equation}
 Then, the boundary term can be managed as
\begin{align}\label{3t}
\int_{\partial\Omega}&(\mbox{curl}J_t\times n)\cdot J_tdS\nonumber\\[2mm]
=&-\int_{\partial\Omega}(\psi \cdot n)|J_t|^2dS-\int_{\partial\Omega}(\psi_t \cdot n)J\cdot J_tdS\nonumber\\[2mm]
=&-\int_{\Omega}\mbox{div}\left(\psi|J_t|^2\right)dx-\int_{\Omega}\mbox{div}\left(\psi_t J\cdot J_t \right)dx\nonumber\\[2mm]
=&-\int_{\Omega}|J_t|^2\mbox{div}\psi dx-2\int_{\Omega}\psi\cdot\nabla J_t\cdot J_t dx
-\int_{\Omega}J\cdot J_t\mbox{div}\psi_t dx\nonumber\\[2mm]
&-\int_{\Omega}\psi_t\cdot\nabla J_t\cdot J dx-\int_{\Omega}\psi_t\cdot\nabla J\cdot J_t dx
\nonumber\\[2mm]
\leq&\varepsilon\|\nabla J_t\|^2+C\left(\|\mbox{div}\psi\|^2_{L^3}+\|\psi\|^2_{L^{\infty}}\right)\|J_t\|^2
+\|\psi_t\|^2_1\|J\|^2_{1}.
\end{align}
The right side of (\ref{2t}) can be restated as follows:
\begin{align}
\int_{\Omega}F_t\cdot J_tdx=&\int_{\Omega}(-v\cdot\nabla J+J\cdot\nabla v+\eta J\mbox{div}\psi+\eta J\psi^2+2\eta \psi\cdot \nabla J)_t\cdot J_tdx\nonumber\\[2mm]
=&\int_{\Omega}(-v_t\cdot\nabla J-v\cdot\nabla J_t+J_t\cdot\nabla v+J\cdot\nabla v_t+
\eta J_t\mbox{div}\psi+ \eta J\mbox{div}\psi_t\nonumber\\[2mm]
&+\eta J_t\psi^2+2\eta J\psi\psi_t
+2\eta \psi_t\cdot \nabla J
+2\eta \psi\cdot \nabla J_t)\cdot J_tdx \equiv \sum_{j=1}^{10}K_{j},\nonumber
\end{align}
where $K_j$ $(1\leq j\leq10)$ can be controlled as
\begin{align}
&K_1=-\int_{\Omega}v_t\cdot\nabla J\cdot J_tdx\leq \|v_t\|_{L^3}\|\nabla J\|\|J_t\|_{L^6}\leq
\varepsilon\|\nabla J_t\|^2+C\|v_t\|^2_{L^3}\|\nabla J\|^2,\nonumber\\[2mm]
&K_2=-\int_{\Omega}v\cdot\nabla J_t\cdot J_tdx\leq\varepsilon\|\nabla J_t\|^2+C\|v\|^2_{L^{\infty}} \| J_t\|^2,\nonumber\\[2mm]
&K_3=\int_{\Omega}J_t\cdot\nabla v\cdot J_tdx\leq \|\nabla v\|_{L^{\infty}}\|J_t\|^2, \nonumber\\[2mm]
&K_4=\int_{\Omega}J\cdot\nabla v_t\cdot J_tdx\leq  \|\nabla v_t\| \|J\|_{L^3}\|J_t\|_{L^6}\leq \varepsilon\|\nabla J_t\|^2+C\|\nabla v_t\|^2 \|J\|^2_{L^3},\nonumber\\[2mm]
&K_5=\eta\int_{\Omega} J_t\mbox{div}\psi\cdot J_tdx\leq\|\mbox{div}\psi\|_{L^3}\|J_t\|\|J_t\|_{L^6}\leq \varepsilon\|\nabla J_t\|^2+C\|\mbox{div}\psi\|^2_{L^3}\|J_t\|^2,\nonumber\\[2mm]
&K_6=\eta\int_{\Omega} J\mbox{div}\psi_t\cdot J_tdx\leq \|\mbox{div}\psi_t\|\|J\|_{L^3}\|J_t\|_{L^6}\leq \varepsilon\|\nabla J_t\|^2+C\|\mbox{div}\psi_t\|^2\|J\|^2_{L^3}, \nonumber\\[2mm]
&K_7=\eta\int_{\Omega}\psi^2 J_t \cdot J_tdx\leq \|\psi\|_{L^{\infty}}^2 \|J_t\|^2, \nonumber\\[2mm]
&K_8=2\eta\int_{\Omega}J\psi\psi_t \cdot J_tdx\leq\|\psi\|_{L^{\infty}}\|\psi_t\|_{L^3}\|J_t\|_{L^6}\|J\|\leq
 \varepsilon\|\nabla J_t\|^2+C\|\psi\|^2_{L^{\infty}}\|\psi_t\|^2_{L^3}\|J\|^2,
\nonumber\\[2mm]
&K_9=2\eta\int_{\Omega}\psi_t\cdot \nabla J\cdot J_tdx\leq \|\psi_t\|_{L^3}\|\nabla J\|\|J_t\|_{L^6} \leq
\varepsilon\|\nabla J_t\|^2+C\|\psi_t\|^2_{L^3}\|\nabla J\|^2,
\nonumber\\[2mm]
&K_{10}=2\eta\int_{\Omega} \psi\cdot \nabla J_t\cdot J_tdx\leq \|\psi\|_{L^{\infty}} \|\nabla J_t\|\|J_t\|\leq
 \varepsilon\|\nabla J_t\|^2+ C\|\psi\|^2_{L^{\infty}}\|J_t\|^2.\nonumber
\end{align}
Therefore, we obtain
\begin{align}\label{Ft}
\int_{\Omega}F_t\cdot J_tdx\leq& \varepsilon\|\nabla J_t\|^2
+C(\|v\|^2_{L^{\infty}}+\|\mbox{div}\psi\|^2_{L^3}+ \|\nabla v\|_{L^{\infty}}+\|\psi\|_{L^{\infty}}^2)\| J_t\|^2\nonumber\\[2mm]
&+C\left(\|v_t\|^2_{1}+\|\psi_t\|^2_{1}\right)\| J\|^2_{1}
+C\|\psi\|^2_{L^{\infty}}\|\psi_t\|^2_{L^3}\|J\|^2.
\end{align}
By substituting (\ref{3t}) and (\ref{Ft}) into (\ref{2t}), then applying (\ref{curl-div}) for $J_t$, one obtains
\begin{align}\label{Jt}
	\frac{d}{dt}\|J_t\|^2+\|\nabla J_t\|^2
	\leq& C\left(1+\|v\|^2_{3}+\|\mbox{div}\psi\|^2_{L^3}+\|\psi\|_{L^{\infty}}^2\right)\|J_t\|^2\nonumber\\[2mm]
&+C\left(\|v_t\|^2_{1}+\|\psi_t\|^2_{1}\right)\| J\|^2_{1}
+C\|\psi\|^2_{L^{\infty}}\|\psi_t\|^2_{L^3}\|J\|^2.
\end{align}
Thanks to
 \begin{equation*}
		\|J_t(0)\|^2\leq C(\|\nabla^2 J(0)\|^2+\|F(0)\|^2)\leq Cc_0^6,
\end{equation*}
then utilizing the Gronwall's inequality to (\ref{Jt}) implies
\begin{align}\label{p-H}
	\begin{split}
\|J_t(t)\|^2+\int_0^t\|\nabla J_t(s)\|^2ds\leq Cc_0^6
\end{split}
\end{align}
 for  $0\leq t\leq T_2= \min\{T^{*},(1+c_{3})^{-6}\}$.
Furthermore, employing \eqref{2J} in Appendix \ref{prop-J}, we obtain
\begin{align}
\|\nabla^2J(t)\|^2\leq C(\|J_t\|^2+\|F\|^2+\| J\|_1^2+\|\psi\|^2_{L^{\infty}} \|J\|^2_1+\|\nabla\psi\|^2_{L^3}\|\nabla J\|^2)\leq C c_1c_2c_0^6,  \nonumber
\end{align}
for $0\leq t\leq T_2$.
Finally, applying the elliptic estimate \eqref{3J} in Appendix \ref{prop-J}, it is obvious that
\begin{align}
\int_0^{t}\|\nabla^3J(s)\|^2ds\leq  C\int_0^t
\left(\|\nabla J_t\|^2+\|\nabla F\|^2+\|J\|^2_{2}+(\|\psi\|^2_{L^{\infty}}+\|\nabla\psi\|^2_{1})\|J\|^2_{2}\right)ds\leq Cc_0^6 \nonumber
\end{align}
for $0\leq t\leq T_2$.
This ends the proof.
\end{proof}

After that, we provide the estimations of the magnetic field $H$.

\begin{lem}\label{lem-H}
Let $(\phi, H, J, u)$ be the unique strong solution to (\ref{linear})-(\ref{bdy-Jl}) on $[0,T]\times \Omega$. Then for $0\leq t\leq T_2= \min\{T^{*},(1+c_{3})^{-6}\}$, it holds
\begin{equation*}
	\begin{split}
&\|H(t)\|^2_1+\int_0^t\left(\|\nabla H\|^2_1+\|H_t\|^2\right)ds\leq Cc_0^2,\nonumber\\[2mm]
&\|H_t(t)\|_1^2+\int_0^t\left(\|\nabla^3 H\|_1^2+\|\nabla H_t\|_1^2+\|H_{tt}\|^2\right)ds\leq C c_0^4,\nonumber\\[2mm]
&\|\nabla^2 H(t)\|^2 \leq Cc_1c_2c_0^2,\quad \|\nabla^3 H(t)\|^2\leq  Cc_1c_2^3c_0^2.
\end{split}
\end{equation*}
\end{lem}
\begin{proof} Since the magnetic field $H$ satisfies the parabolic equation with homogeneous boundary conditions, we can obtain a prior estimate by the standard energy estimates argument for parabolic equations. Thus we omit it here.
\end{proof}

At last, we would like to estimate the velocity  $u$.
\begin{lem}\label{lem-u}
Let $(\phi,u, H, J)$ be the unique strong solution to (\ref{linear})-(\ref{bdy-Jl}) on $[0,T]\times \Omega$. Then for $0\leq t\leq T_3=\min\{T^{*},(1+c_{3})^{-9}\}$, it holds out that
\begin{align}
		&\|u(t)\|^2_1+\int_0^t\left(\|\nabla u\|^2_1+\|u_t\|^2\right)ds\leq Cc_0^2,\nonumber\\[2mm]
&\|u_t(t)\|^2+\int_0^t\left(\|\nabla^3u\|^2+\|\nabla u_t\|^2\right)ds\leq C c_0^4,\nonumber\\[2mm]
		&\|\nabla u_t(t)\|^2+\int_0^t\left(\|\nabla^4u\|^2+\|\nabla^2 u_t\|^2+\|u_{tt}\|^2\right)ds\leq C c_0^4,\nonumber\\[2mm]
		&\|\nabla^2u(t)\|^2\leq Cc_0^4c_1^3c_2,\quad \|\nabla^3u(t)\|^2\leq Cc_0^4c_1c_2^3.\nonumber
		\end{align}
\end{lem}
\begin{proof}
Recall the  linear equation of  $(\ref{linear})_2$ as
\begin{equation}\label{equ-u}
 u_{t}+\alpha\mbox{curl}^2 u-(2\alpha+\beta)\nabla\mbox{div}u
=-v\cdot\nabla v-\frac{2A\gamma}{\gamma-1}\phi\nabla\phi+\psi Q(v)+J\cdot\nabla H-J\cdot(\nabla H)^{T}\equiv N,
\end{equation}
with the boundary conditions:
\begin{equation}\label{boundary-u}
 	u\cdot n=0,\quad  \mbox{curl}u\times n=-K(x)u,\quad \mbox{on}\quad \partial\Omega.
 \end{equation}
{\it Step 1:} We will first calculate $\|u\|$.
Multiplying the equation (\ref{equ-u}) by $u$,  integrating the result identity over $\Omega$, and  performing integration by parts with the boundary conditions, it follows that
\begin{align}\label{1orderu}
\frac{1}{2}\frac{d}{dt}\|u\|^2+\alpha\|\mbox{curl}u\|^2+(2\alpha+\beta)\|\mbox{div}u\|^2
=-\alpha\int_{\partial\Omega}K(x)|u|^2dS+\int_{\Omega}N\cdot u dx.
\end{align}
Using the trace estimate, the boundary term in \eqref{1orderu} can be governed as
\begin{align}\label{trace1}
-\alpha\int_{\partial\Omega}K(x) |u|^2dS&\leq C\left(\|\nabla u\|\|u\|+\|u\|^2\right)\nonumber\\[2mm]
&\leq \varepsilon \|\nabla u\|^2 +C\|u\|^2.
\end{align}
By applying Cauchy's inequality, one possesses
\begin{equation}
\int_{\Omega}N\cdot u\leq \|u\|^2+\|N\|^2, \nonumber
\end{equation}
where
\begin{align}\label{N}
	\|N\|^2&\leq C\left(\|v\|^2_{L^{\infty}}\|\nabla v\|^2+\|\phi\|^2_{L^{\infty}}\|\nabla\phi\|^2+
	\|\psi\|^2_{L^{\infty}}\|Q(v)\|^2+\| J\|_{L^6}^2\|\nabla H\|_{L^3}^2\right)\nonumber\\[2mm]
	&\leq C \left(\|\nabla v\|^3\|\nabla^2 v\|+\|\nabla \phi\|^3\|\nabla^2 \phi\|+	 \|\psi\|^2_{L^{\infty}}\|\nabla v\|^2+\|\nabla J\|^2\|\nabla H\|_1^2\right).
	\end{align}
Substituting the above inequalities into (\ref{1orderu}) and applying (\ref{curl-div}) implies
	\begin{align}
		\frac{1}{2}\frac{d}{dt}\|u\|^2+\|\nabla u\|^2\leq C\|u\|^2+C\left(\|\nabla v\|^3\|\nabla^2 v\|+\|\nabla \phi\|^3\|\nabla^2 \phi\|+	 \|\psi\|^2_{L^{\infty}}\|\nabla v\|^2+\|\nabla J\|^2\|\nabla H\|_1^2\right).\nonumber
	\end{align}
By using the Gronwall's inequality, for $0\leq t\leq T_2= \min\{T^{*},(1+c_{3})^{-6}\}$, it is clear that
\begin{align}\label{8}
\|u(t)\|^2+\int_0^t\|\nabla u\|^2ds \leq C(\|u_0\|^2+c_0^4c_3^4t)\exp(Ct)\leq Cc_0^2.
\end{align}
{\it Step 2:} This procedure will yield the estimate $\|\nabla u\|$.
We multiply the equation (\ref{equ-u}) by \( u_t \), integrate the resulting identity over \( \Omega \), and apply integration by parts with respect to spatial variable to demonstrate
\begin{equation}
 	\begin{split}
	\frac{1}{2}\frac{d}{dt}\left(\alpha\|\mbox{curl} u\|^2+(2\alpha+\beta)\|\mbox{div}u\|^2\right)+\|u_t\|^2
&=\alpha\int_{\partial\Omega} (\mbox{curl}u\times n)\cdot u_tdS+\int_{\Omega}N\cdot u_t\nonumber\\[2mm]
&=-\alpha\int_{\partial\Omega} K(x)u \cdot u_tdS+\int_{\Omega}N\cdot u_t\nonumber\\[2mm]
&=-\frac{\alpha}{2}\frac{d}{dt}\int_{\partial\Omega} K(x)|u|^2dS+\int_{\Omega}N\cdot u_t\nonumber\\[2mm]
	&\leq -\frac{\alpha}{2}\frac{d}{dt}\int_{\partial\Omega} K(x)|u|^2dS+\frac{1}{2}\|u_t\|^2+C\|N\|^2,\nonumber
\end{split}
\end{equation}
which yields
\begin{equation}
 	\frac{d}{dt}\left(\alpha\|\mbox{curl} u\|^2+(2\alpha+\beta)\|\mbox{div}u\|^2\right)
+\alpha\frac{d}{dt}\int_{\partial\Omega} K(x)|u|^2dS+\|u_t\|^2\leq C\|N\|^2. \nonumber
\end{equation}
Integrating the above over $s\in[0,t]$, and using inequalities (\ref{curl-div}), (\ref{N}) and (\ref{8}), it follows
 \begin{align}\label{10}
 &\|\nabla u(t)\|^2+\int_{\partial\Omega} K(x)|u|^2dS+\int_0^t\|u_t\|^2ds\nonumber\\[2mm]
 &\leq  C\int_{\partial\Omega} K(x)|u(0)|^2dS+	\|\nabla u(0)\|^2+\|u(t)\|^2+C\|N\|^2t\nonumber\\[2mm]
 &\leq C\|u(0)\|_1^2+	\|u(t)\|^2+Cc_0^4c_3^4t\nonumber\\[2mm]
 &\leq Cc_0^2,
\end{align}
for $0\leq t\leq T_2= \min\{T^{*},(1+c_{3})^{-6}\}$. Moreover, \eqref{10} implies that
 \begin{align}\label{10-2}
 &\|\nabla u(t)\|^2+\int_0^t\|u_t\|^2ds\leq Cc_0^2.
\end{align}
where we have used \eqref{trace1} for the boundary integral. Furthermore, employing the classical elliptic estimate to $u$, we can find that
\begin{equation*}
	\begin{split}
	\int_0^t\|\nabla^2 u\|^2ds\leq C\int_0^t(\|u_t\|^2+\|N\|^2+\|\nabla u\|^2)ds
	\leq Cc_0^2.
	\end{split}
\end{equation*}

{\it Step 3:} The estimate $\|\nabla^2 u\|$ will be got in this step. Taking the derivative of (\ref{equ-u}) with respect to $t$, and multiplying the result by $u_{t}$,  integrating by parts with the boundary conditions $u_t\cdot n|_{\partial\Omega}=0$ and $\mbox{curl}u_t\times n=-K(x)u_t$, it results in
 \begin{align}\label{7}
&\frac{1}{2}\frac{d}{dt}\|u_t\|^2+\alpha\|\mbox{curl} u_t\|^2+(2\alpha+\beta)\|\mbox{div}u_t\|^2=-\alpha\int_{\partial\Omega}K(x) |u_t|^2dS+\int_{\Omega}N_t\cdot u_t.
\end{align}
Using the trace estimate,
it holds
\begin{equation}\label{bound-ut}
-\alpha\int_{\partial\Omega}K(x) |u_t|^2dS
\leq \varepsilon \|\nabla u_t\|^2 +C\|u_t\|^2.
\end{equation}
Employing H\"{o}lder's inequality, Cauchy's inequality and Sobolev inequalities, it is clear that
\begin{align}\label{u-1}
	\int_{\Omega}N_t\cdot u_t
\leq& \varepsilon\|\nabla u_t\|^2+C\|u_t\|^2+C\left(\| v\|_2^2\|v_t\|_1^2+\|\phi\|^2_{L^\infty}\|\phi_t\|^2+\|\psi\|_{L^{\infty}}^2\|\nabla v_t\|^2\right)\nonumber\\[2mm]
	&+C\left(\|\psi_t\|^2\|\nabla v\|_{L^3}^2+\| J_t\|^2\|\nabla H\|_{L^3}^2+\| J\|_{L^3}^2\|\nabla H_t\|^2\right),
	\end{align}
where we have used the fact
\begin{equation*}
	\begin{split}
	-\frac{2A\gamma}{\gamma-1}\int_{\Omega} (\phi\nabla\phi)_tu_t&=\frac{A\gamma}{\gamma-1}\int_{\Omega} (\phi^2)_t\mbox{div}u_t\leq C\|\phi_t\|\|\phi\|_{L^\infty}\|\nabla u_t\|\\[2mm]
	&\leq \varepsilon\|\nabla u_t\|^2+C\|\phi\|^2_{L^\infty}\|\phi_t\|^2.
	\end{split}
\end{equation*}
Then, adding \eqref{bound-ut}, \eqref{u-1} into \eqref{7} and  applying Lemma \ref{lem-div-curl}, choosing suitable small $\varepsilon$,  we have
 \begin{align}\label{u-2}
\frac{d}{dt}\|u_t\|^2+\|\nabla u_t\|^2
\leq& C\|u_t\|^2+C\left(\| v\|_2^2\|v_t\|_1^2+\|\phi\|^2_{L^\infty}\|\phi_t\|^2+\|\psi\|_{L^{\infty}}^2\|\nabla v_t\|^2\right)\nonumber\\[2mm]
	&+C\left(\|\psi_t\|^2\|\nabla v\|_{L^3}^2+\| J_t\|^2\|\nabla H\|_{L^3}^2+\| J\|_{L^3}^2\|\nabla H_t\|^2\right).
\end{align}
According to the equation $(\ref{linear})_2$, it is easy to see that
\begin{align}
\|u_t(0)\|^2\leq Cc_0^4.\nonumber
\end{align}
Thus, applying the Gronwall's inequality to \eqref{u-2},  it means
\begin{equation}\label{dut}
	\|u_t\|^2+\int_0^t\|\nabla u_t\|^2ds\leq Cc_0^4
\end{equation}
for $0\leq t\leq T_2= \min\{T^{*},(1+c_{3})^{-6}\}$. Moreover, the following inequalities are obtained
\begin{equation}
	\|\nabla^2 u\|^2\leq C(\|u_t\|^2+\|N\|^2+\|\nabla u\|^2)\leq C(c_0^4+c_0^4c_1^3c_2)\leq Cc_0^4c_1^3c_2,\nonumber
	\end{equation}
\begin{equation}
	\int_0^t	\|\nabla^3 u\|^2ds\leq C\int_0^t(\|\nabla u_t\|^2+\|\nabla N\|^2+\|\nabla u\|^2)ds
\leq Cc_0^4,\nonumber
	\end{equation}
by the classical elliptic estimate and
\begin{align}\label{1orderN}
	\|\nabla N\|^2\leq& C\left(\|\nabla v\|\|\nabla^2 v\|^3+\|\nabla \phi\|\|\nabla^2 \phi\|^3+\|\nabla\psi\|_1^2\|\nabla v\|_1^2\right)\nonumber\\
	&+\|\nabla J\|^2\left(\|\nabla H\|\|\nabla^2 H\|+\|\nabla ^2H\|\|\nabla^3 H\|\right)\nonumber\\[2mm]
	\leq &Cc_0^4c_1c_2^3.
	\end{align}\\
{\it Step 4:} \ Estimate on $\|\nabla^3u\|$. Taking $\int_{\Omega}\partial_{t}(\ref{equ-u})\cdot u_{tt}$, applying the boundary conditions for $u_t$, it is evident that
 \begin{equation}
 	\begin{split}
	\frac{1}{2}&\frac{d}{dt}\left(\alpha\|\mbox{curl} u_t\|^2+(2\alpha+\beta)\|\mbox{div}u_t\|^2\right)+\|u_{tt}\|^2
=\alpha\int_{\partial\Omega} (\mbox{curl}u_t\times n)\cdot u_{tt}dS+\int_{\Omega}N_t\cdot u_{tt}\nonumber\\[2mm]
&=-\alpha\int_{\partial\Omega} K(x)u_t \cdot u_{tt}+\int_{\Omega}N_t\cdot u_{tt}\nonumber\\[2mm]
&=-\frac{\alpha}{2}\frac{d}{dt}\int_{\partial\Omega} K(x)|u_t|^2dS+\int_{\Omega}N_t\cdot u_{tt}\nonumber\\[2mm]
	&\leq -\frac{\alpha}{2}\frac{d}{dt}\int_{\partial\Omega} K(x)|u_t|^2dS+\frac{1}{2}\|u_{tt}\|^2+C\|N_t\|^2,\nonumber
\end{split}
\end{equation}
which implies
\begin{equation}
 	\begin{split}
\frac{d}{dt}\left(\alpha\|\mbox{curl} u_t\|^2+(2\alpha+\beta)\|\mbox{div}u_t\|^2\right)+\alpha\frac{d}{dt}\int_{\partial\Omega} K(x)|u_t|^2dS+\|u_{tt}\|^2\leq C\|N_t\|^2.\nonumber
\end{split}
\end{equation}
Therefore, integrating the above inequality over $s\in[0,t]$ and using Lemma \ref{lem-div-curl} for $u_t$ with $u_t\cdot n|_{\partial\Omega}=0$ means that
\begin{align}\label{utt-1}
&\|\nabla u_t\|^2+\int_{\partial\Omega} K(x)|u_t|^2dS+\int_0^t\|u_{tt}\|^2ds\nonumber\\[2mm]
&\leq C \|\nabla u_t(0)\|^2+C\int_{\partial\Omega} K(x)|u_t(0)|^2dS+C\|u_t\|^2+C\int_{0}^t\|N_t\|^2ds.
\end{align}
We have
\begin{align}\label{initial-1ut}
\|\nabla u_t(0)\|^2\leq Cc_0^4
\end{align}
due to \eqref{equ-u}. For the nonlinear terms, it is easy to see that
\begin{align}\label{Nt}
\|N_t\|^2\leq&	\|v_t\|^2_{L^6}\|\nabla v\|^2_{L^{3}}+\|v\|^2_{L^{\infty}}\|\nabla v_t\|^2+\|\phi_t\|^2_{L^6}\|\nabla \phi\|^2_{L^3}
+\|\phi\|^2_{L^{\infty}}\|\nabla \phi_t\|^2\nonumber\\[2mm]
&+\|\psi_t\|^2_{L^6}\|\nabla v\|^2_{L^3}
+\|\psi\|^2_{L^{\infty}}\|\nabla v_t\|^2
+\|J_t\|^2\|\nabla H\|^2_{L^{\infty}}
+\|J\|^2_{L^{\infty}}\|\nabla H_t\|^2.
		\end{align}
Combining  \eqref{initial-1ut}, \eqref{Nt} with (\ref{utt-1}), and using \eqref{bound-ut}, it holds
\begin{equation}\label{utt-2}
\|\nabla u_t\|^2+\int_0^t\|u_{tt}\|^2ds
\leq Cc_0^4.
\end{equation}
for $0\leq t\leq \widetilde{T_3}=\min\{T^{*},(1+c_{3})^{-8}\}$. Next, the classical elliptic estimate induces that
\begin{equation}
	\|\nabla^3 u\|^2\leq C\left(\|\nabla u_t\|^2+\|\nabla N\|^2+\|\nabla u\|^2\right)\leq Cc_0^4c_1c_2^3.\nonumber
\end{equation}
Thanks to the classical elliptic estimate, using (\ref{dut}), (\ref{Nt}) and (\ref{utt-2}), we have
\begin{equation*}
	\begin{split}
		\int_0^t\|\nabla^2u_t\|^2ds\leq &C\int_0^t\left(\|u_{tt}\|^2+\|N_t\|^2+\|\nabla u_t\|^2\right)ds\\[2mm]
		\leq &C(c_0^4+c_0^8c_1c_2^2t)\leq Cc_0^4,
		\end{split}
	\end{equation*}
 for $0\leq t\leq \widetilde{T_3}=\min\{T^{*},(1+c_{3})^{-8}\}$, which together with
 \begin{equation*}
 	\begin{split}
 		\|\nabla^2N\|^2\leq C\left(\|\nabla v\|_2^4+\|\nabla \phi\|_2^4+\|\nabla^2 v\|_1^2\|\nabla\psi\|_1^2+\|\nabla J\|_1^2\|\nabla^2 H\|_1^2\right)\leq Cc_3^{13}.
 		\end{split}
 	\end{equation*}
yields that
\begin{equation}
	\begin{split}
		\int_0^t\|\nabla^4u\|^2ds\leq &C\int_0^t\left(\|\nabla^2u_{t}\|^2+\|\nabla^2N\|^2+\|\nabla u\|^2\right)ds\\[2mm]
		\leq &C(c_0^4+c_3^{13}t)\leq Cc_0^4
	\end{split}
\end{equation}
for $0\leq t\leq T_3=\min\{T^{*},(1+c_{3})^{-9}\}$. Therefore, the proof of Lemma \ref{lem-u}  is completed.
\end{proof}

Then combining the estimates obtained in Lemmas \ref{lem-phi}-\ref{lem-u}, we have
\begin{align*}
&\|\phi(t)\|^2_{\infty}+\|\phi(t)-\phi^{\infty}\|^2_3+\|\phi_t(t)\|^2_{2}
\leq Cc_0^2c_3^4,\\[2mm]
&\|\psi(t)\|^2_{\infty}+\|\nabla \psi(t)\|^2_{1}+\|\psi_t(t)\|^2_{1}\leq C c_0^2c_3^2,\\[2mm]
&\|H(t)\|^2_1+\int_0^t\left(\|\nabla H\|^2_1+\|H_t\|^2\right)ds\leq Cc_0^2,\nonumber\\[2mm]
&\|H_t(t)\|^2+\|\nabla^2 H(t)\|^2+\int_0^t\left(\|\nabla^3 H\|^2+\|\nabla H_t\|^2\right)ds\leq C c_1c_2c_0^2,\nonumber\\[2mm]
&\|\nabla H_t(t)\|^2+\|\nabla^3 H(t)\|^2+\int_0^t\left(\|\nabla^4 H\|^2+\|\nabla^2 H_t\|^2+\|H_{tt}\|^2\right)ds\leq Cc_1c_2^3c_0^2,\nonumber\\[2mm]
&\|J\|^2+\int_{0}^{t}\|\nabla J\|^2ds\leq C c^2_0,\nonumber\\[2mm]
&\|\nabla J\|^2+\int_0^t(\|J_t\|^2+\|\nabla^2J\|^2)ds\leq C c_0^4,\nonumber\\[2mm]
&\|J_t\|^2+\|\nabla^2J\|^2+\int_0^t(\|\nabla J_t\|^2+\|\nabla^3 J\|^2)ds\leq Cc_1c_2 c_0^6,\nonumber\\[2mm]
&\|u(t)\|^2_1+\int_0^t\left(\|\nabla u\|^2_1+\|u_t\|^2\right)ds\leq Cc_0^2,\\[2mm]
		&\|u_t(t)\|^2+\|\nabla^2u(t)\|^2+\int_0^t\left(\|\nabla^3u\|^2+\|\nabla u_t\|^2\right)ds\leq C c_0^4c_1^3c_2,\\[2mm]
		&\|\nabla u_t(t)\|^2+\|\nabla^3u(t)\|^2+\int_0^t\left(\|\nabla^4u\|^2+\|\nabla^2 u_t\|^2+\|u_{tt}\|^2\right)ds\leq Cc_0^4c_1c_2^3,
\end{align*}
for $0\leq t\leq T_3=\min\{T^{*},(1+c_{3})^{-9}\}$. Therefore, if we define the time
and constants as
\begin{align}\label{constant}
	&c_1=C^{\frac{1}{2}}c_0,\quad c_2=Cc_0^4c_1^3=C^{\frac{5}{2}}c_0^7,\nonumber\\ &c_3=C^{\frac{1}{2}}c_0^2c_1^{\frac{1}{2}}c_2^{\frac{3}{2}}=C^{\frac{9}{2}}c_0^{13},
	\quad T^*=\min\{T, (1+c_3)^{-9}\},
	\end{align}
then  we can deduce that
\begin{align}\label{local-linear}
&\sup_{0\leq t\leq T^*}\|u(t)\|^2_1+\int_0^t\left(\|\nabla u\|^2_1+\|u_t\|^2\right)ds\leq c_1^2,\nonumber\\[2mm]
		&\sup_{0\leq t\leq T^*}\left(\|u_t(t)\|^2+\|\nabla^2u(t)\|^2\right)+\int_0^t\left(\|\nabla^3u\|^2+\|\nabla u_t\|^2\right)ds\leq c_2^2,\nonumber\\[2mm]
		&\sup_{0\leq t\leq T^*}\left(\|\nabla u_t(t)\|^2+\|\nabla^3u(t)\|^2\right)+\int_0^t\left(\|\nabla^4u\|^2+\|\nabla^2 u_t\|^2+\|u_{tt}\|^2\right)ds\leq c_3^2,\nonumber\\[2mm]
&\sup_{0\leq t\leq T^*}\left(\|\phi(t)\|^2_{\infty}+\|\phi(t)-\phi^{\infty}\|^2_3+\|\phi_t\|_2^2\right)
\leq c_3^6,\nonumber\\[2mm]
&\sup_{0\leq t\leq T^*}\left(\|\psi(t)\|^2_{\infty}+\|\nabla\psi(t)\|^2_{1}+\|\psi_t\|_1^2\right)
\leq c_3^4,\nonumber\\[2mm]
&\sup_{0\leq t\leq T^*}\|H(t)\|^2_1+\int_0^t\left(\|\nabla H\|^2_1+\|H_t\|^2\right)ds\leq c_1^2,\nonumber\\[2mm]
&\sup_{0\leq t\leq T^*}\left(\|H_t(t)\|^2+\|\nabla^2 H(t)\|^2\right)+\int_0^t\left(\|\nabla^3 H\|^2+\|\nabla H_t\|^2\right)ds\leq c_2^2,\nonumber\\[2mm]
&\sup_{0\leq t\leq T^*}\left(\|\nabla H_t(t)\|^2+\|\nabla^3 H(t)\|^2\right)+\int_0^t\left(\|\nabla^4 H\|^2+\|\nabla^2 H_t\|^2+\|H_{tt}\|^2\right)ds\leq c_3^2,\nonumber\\[2mm]
&\sup_{0\leq t\leq T^*}\|J\|^2+\int_{0}^{t}\|\nabla J\|^2ds\leq  c^2_1,\nonumber\\[2mm]
&\sup_{0\leq t\leq T^*}\|\nabla J\|^2+\int_0^t(\|J_t\|^2+\|\nabla^2 J\|^2)ds\leq  c_2^2,\nonumber\\[2mm]
&\sup_{0\leq t\leq T^*}\left(\|J_t\|^2+\|\nabla^2J\|^2\right)+\int_0^t(\|\nabla J_t\|^2+\|\nabla^3J\|^2)ds\leq  c_3^2.
\end{align}
In conclusion, given fixed constants $c_0$ and $T$, there really exists constants $T^{*}$ and $c_i$, $(i=1,2,3)$, which depend only on $c_0$ and $T$, such that if $v$ satisfy the conditions (\ref{known}), then the strong  solution $(\phi, H,u,J)$ to the initial-boundary value problem (\ref{linear})- (\ref{bdy-Jl}) not only exists for $\Omega\times[0,T^*]$, but also satisfies  the estimates (\ref{local-linear}).
In the proof  of Theorem \ref{thm2}, these estimates play an important role in the iteration process.

\subsection{Unique solvability of the linearization with vacuum}\label{3.4}
Based on the local (in time) estimates in (\ref{local-linear}), we have the following existence result under the assumption that $\phi_0>0$.
\begin{lem}\label{lem4}
Assume that the initial data $(\phi_0,u_0,H_0, J_0)$ satisfy (\ref{intial2}). Then for $T^*>0$, there exists a unique strong solution $(\phi,u,H,J)$ to the initial-boundary value problem (\ref{linear})-(\ref{bdy-Jl}) such that
\begin{equation}
	\begin{split}
		&\phi\geq 0,\, \phi\in C([0,T^*];H^3), \,\phi_t\in C([0,T^*];H^2),\\[2mm]
		&\psi\in  C([0,T^*];D^1\cap D^2), \psi_t\in  C([0,T^*];H^1),\\[2mm]
		&H\in C([0,T^*];H^3)\cap L^2([0,T^*];H^4), H_t\in C([0,T^*];H^1)\cap L^2([0,T^*];H^2),\\[2mm]
		&H_{tt}\in L^2([0,T^*];L^2),\\[2mm]
		&J\in L^{\infty}([0,T^*];H^2)\cap L^2([0,T^*];H^3), J_t\in L^{\infty}([0,T_{*}];L^2)\cap L^{2}(0,T^{*};H^1),\\[2mm]
		&u\in C([0,T^*];H^3)\cap L^2([0,T^*];H^4), u_t\in C([0,T^*];H^1)\cap L^2([0,T^*];H^2),\\[2mm]
		&u_{tt}\in L^2([0,T^*];L^2).
		\end{split}
	\end{equation}
Moreover, the solution $(\phi,\psi, H,J,u)$ also satisfies the estimates in (\ref{local-linear}).
\end{lem}
\begin{proof}
	For $\sigma\in(0,1)$, we define that
	\begin{equation*}
		\phi_{\sigma0}=\phi_0+\sigma, \quad\psi_{\sigma0}=\frac{2\nabla\phi_0}{(\gamma-1)(\phi_0+\sigma)}, \quad J_{\sigma0}=\frac{H_0}{(\phi_0+\sigma)^{\frac{2}{\gamma-1}}}.
		\end{equation*}
	Since $\psi_0\in D^1\cap D^2$, by employing the Gagliardo-Nirenberg inequality, there exists a constant $C_1$ such that for any $x\in\Omega$, the following holds
	\begin{equation*}
		 \left\|\frac{\nabla\phi_0}{\phi_0}\right\|_{L^{\infty}}=\frac{\gamma-1}{2}\|\psi_0\|_{L^{\infty}}\leq C\|\nabla\psi_0\|^{\frac{1}{2}}\|\nabla^2\psi_0\|^{\frac{1}{2}}\leq Cc_0,
		\end{equation*}
	which implies that
	\begin{equation*}
|\nabla\phi_0|\leq C_1\phi_0.
\end{equation*}
Hence we have $\nabla\phi_0(x)=0$ if $\phi_0(x)=0$. Noted that
\begin{center} $\psi_{\sigma0}-\psi_0=-\frac{\sigma}{\sigma+\phi_0}\psi_0$,\quad $J_{\sigma0}-J_0=-\frac{(\phi_0+\sigma)^{\frac{2}{\gamma-1}}-\phi_0^{\frac{2}{\gamma-1}}}{(\phi_0+\sigma)^{\frac{2}{\gamma-1}}}J_0$,
	\end{center}
	 and according to (\ref{initial-0}), there is a positive constant $\sigma_1$ small enough such that if $0< \sigma< \sigma_1$, then one has
\begin{equation}
1+\|\phi_{\sigma0}\|_{L^{\infty}}+\|\phi_{\sigma0}-\sigma\|_3+\|u_0\|_3+\|H_0\|_3+\|J_{\sigma0}\|_2+|\psi_{\sigma0} |_{ D^1\cap D^2}\leq Cc^4_0=\bar{c}_0.\nonumber
	\end{equation}
Therefore, if we take $(\phi_{\sigma0}, u_0,H_0,J_{\sigma0})$ as the initial data, then the initial-boundary value problem (\ref{linear})-(\ref{bdy-Jl}) exists a unique strong solution $(\phi^{\sigma}, u^{\sigma},H^{\sigma},J^{\sigma})$, which satisfies the estimates (\ref{local-linear}). Because the estimates (\ref{local-linear}) are independent of $\sigma$, one can find a subsequence of solutions (still denoted by) $(\phi^{\sigma},\psi^{\sigma},u^{\sigma},H^{\sigma},J^\sigma)$ that converges to $(\phi,\psi,u,H,J)$ in the weak$^*$ sense as $\sigma\rightarrow0$, that is:
	\begin{equation}\label{weak}
		\begin{split}
			&w^*-\lim_{\sigma\rightarrow0}(\phi^\sigma, u^{\sigma},H^\sigma)=(\phi,u,H)\in L^\infty(0,T^*;H^3);\\[2mm]
&w^*-\lim_{\sigma\rightarrow0}J^\sigma=J\in L^\infty(0,T^*;H^2);\\[2mm]
&w^*-\lim_{\sigma\rightarrow0}J_t^\sigma=J_t\in L^\infty(0,T^*;L^2);\\[2mm]
			&w^*-\lim_{\sigma\rightarrow0}\psi^\sigma=\psi\in L^\infty(0,T^*;D^1\cap D^2);\\[2mm]
			&w^*-\lim_{\sigma\rightarrow0}\phi_t^\sigma=\phi_t\in L^\infty(0,T^*;H^2);\\[2mm]
			&w^*-\lim_{\sigma\rightarrow0}(\psi_t^{\sigma},u_t^{\sigma},H_t^\sigma)=(\psi_t,u_t,H_t)\in L^\infty(0,T^*;H^1).
			\end{split}
		\end{equation}
	
In addition, for any constant $R>0$, employing the uniform estimates (\ref{local-linear}) and Aubin-Lions Lemma (refer to Lemma \ref{lem-AL}), there  exists  a subsequence of solutions (still denoted by) $(\phi^{\sigma},\psi^{\sigma}, u^{\sigma},H^{\sigma},J^{\sigma})$ satisfying
	\begin{equation}\label{strong}
		\lim_{\sigma\rightarrow0}(\phi^\sigma,\psi^{\sigma},u^{\sigma},H^{\sigma},J^\sigma)=(\phi,\psi, u,H,J)\in C\left([0,T^*];H^1(B_R\cap\Omega)\right)
		\end{equation}
where $B_R=\{x\in\mathbb{R}^3|\,|x|\leq R\}$.
Then through the uniform estimates (\ref{local-linear}), the weak or weak$^*$ convergence in (\ref{weak}) along with the strong convergence in (\ref{strong}), we conclude that $(\phi,u,H,J)$  is a weak solution in the sense of distribution to the linearized problem (\ref{linear})-(\ref{bdy-Jl}), and satisfies the following regularities:
\begin{equation}\label{reg}
	\begin{split}
		&\phi\geq 0, \phi\in L^{\infty}([0,T^*];H^3), \phi_t\in L^{\infty}([0,T^*];H^2),\\[2mm]
		&\psi\in  L^{\infty}([0,T^*];D^1\cap D^2), \psi_t\in  L^{\infty}([0,T^*];H^1),\\[2mm]
		&H\in L^{\infty}([0,T^*];H^3)\cap L^2([0,T^*];H^4), H_t\in L^{\infty}([0,T^*];H^1)\cap L^2([0,T^*];H^2),\\[2mm]
		&H_{tt}\in L^2([0,T^*];L^2),J\in L^{\infty}([0,T^*];H^2)\cap L^2([0,T^*];H^3),\\[2mm]
		&J_t\in L^{\infty}([0,T^*];L^2)\cap L^2([0,T^*];H^1),\\[2mm]
		&u\in L^{\infty}([0,T^*];H^3)\cap L^2([0,T^*];H^4), u_t\in L^{\infty}([0,T^*];H^1)\cap L^2([0,T^*];H^2),\\[2mm]
		&u_{tt}\in L^2([0,T^*];L^2).
	\end{split}
	\end{equation}

The uniqueness and time continuity for $(\phi,\psi,H,u)$ can be obtained by standard procedure, so we omit the details here.
\end{proof}

\subsection{Proof of Theorem \ref{thm2} }\label{3.5}
Under the help of the classical iteration and the  existence results of the  linearization with vacuum of Lemma \ref{lem4}, we can prove that the reformulated
problem (\ref{J})-(\ref{initial-2}) exists a local solution in this subsection.

{\bf {Proof of Theorem \ref{thm2}.}}
Firstly, we  fix a positive constant $c_0$ sufficiently large such that
\begin{align}
		2+\phi^{\infty}+\|\phi_0-\phi^{\infty}\|_3+\|u_0\|_3+\|H_0\|_3+\|J_0\|_2+|\psi_0 |_{D^1\cap D^2}\leq c_0.\nonumber
	\end{align}
Then let $u^0$ be the unique solution to the following linear parabolic problem:
\begin{equation*}
\left\{
\begin{array}{lll}
h_t-\Delta h=0\quad \mbox{in}\quad \Omega\times(0,\infty),\\
h|_{t=0}=u_0\quad \mbox{in}\quad \Omega,\\
h\cdot n=0,\quad \mbox{curl}h\times n=-K(x)h\quad \mbox{on}\quad \partial\Omega,\\
h\rightarrow 0 \quad\mbox{as}\quad |x|\rightarrow \infty, \quad t>0
\end{array}
\right.
\end{equation*}
It is clear that $u^0\in C([0,T^{*}]; H^3)\cap L^2([0,T^*];H^4)$, which implies that there exists a small
time $T^{**}\in (0,T^{*}]$ such that
\begin{align*}\label{known}
	&\sup_{0\leq t\leq T^{**}}\|u^0(t)\|^2_1+\int_0^{T^{**}}\left(\|\nabla u^0\|^2_1+\|u^0_t\|^2\right)dt\leq c_1^2,\nonumber\\[2mm]
	&\sup_{0\leq t\leq T^{**}}\left(\|\nabla^2u^0(t)\|^2+\|u^0_t\|^2\right)+\int_0^{T^{**}}\left(\|\nabla^3u^0\|^2+\|\nabla u^0_t\|^2\right)dt\leq c_2^2,\nonumber\\[2mm]
	&\sup_{0\leq t\leq T^{**}}\left(\|\nabla^3u^0(t)\|^2+\|\nabla u^0_t\|^2\right)+\int_0^{T^{**}}\left(\|\nabla^4u^0\|^2+\|\nabla^2u^0_t\|^2+\|u^0_{tt}\|^2\right)dt\leq c_3^2.
\end{align*}

At the beginning step of iteration, we choose  $v=u^0$, then
 $(\phi^1, \psi^1,H^1, J^1, u^1)$ could be obtained as a strong solution to the linearized problem \eqref{linear}-\eqref{linear2}. Then we construct approximate solutions $(\phi^{k+1}, \psi^{k+1}, H^{k+1}, J^{k+1}, u^{k+1})$ inductively. In other words, given $(\phi^k, \psi^{k},H^k, J^k, u^k)$ is well-defined for $k\geq 1$,
then $(\phi^{k+1}, \psi^{k+1}, H^{k+1}, J^{k+1}, u^{k+1})$ should be defined by solving the problem \eqref{linear}-\eqref{linear2} with $v$ replaced by $u^k$ as follows:
\begin{equation}\label{inter-0}
\begin{split}
\left\{
\begin{array}{llll}
\phi^{k+1}_t+u^{k}\cdot\nabla\phi^{k+1}+\frac{\gamma-1}{2}\phi^{k+1}\mbox{div}u^{k}=0,\\[2mm]
\psi^{k+1}_t+\nabla(u^k\cdot\psi^{k+1})+\nabla \mbox{div}u^{k}=0,\\[2mm]
H^{k+1}_{t}- \eta \Delta H^{k+1}=-u^{k}\cdot\nabla H^{k+1}+H^{k+1}\cdot\nabla u^{k}-H^{k+1}\mbox{div}u^{k},\\[2mm]
J^{k+1}_t-\eta\Delta J^{k+1}=-u^{k}\cdot\nabla J^{k+1}+J^{k+1}\cdot\nabla u^{k}+\eta J^{k+1}\mbox{div}\psi^{k+1}+\eta J^{k+1}(\psi^{k+1})^2
+2\eta \psi^{k+1}\cdot \nabla J^{k+1},\\[2mm]
u^{k+1}_{t}+\alpha\mbox{curl}^2 u^{k+1}-(2\alpha+\beta)\nabla\mbox{div}u^{k+1}\\[1mm]
=-u^{k}\cdot\nabla u^{k}-\frac{2A\gamma}{\gamma-1}\phi^{k+1}\cdot\nabla\phi^{k+1}
+\psi^{k+1}\cdot Q(u^{k})+J^{k+1}\cdot\nabla H^{k+1}-J^{k+1}\cdot(\nabla H^{k+1})^{T}\\[2mm]
(\phi^{k+1},\psi^{k+1},H^{k+1},J^{k+1},u^{k+1})|_{t=0}=(\phi_0,\psi_0,H_0, J_0,u_0)\\[2mm]
u^{k+1}\cdot n|_{\partial\Omega}=0,\quad \mbox{curl}u^{k+1}\times n|_{\partial\Omega}=-K(x)u^{k+1},\quad H^{k+1}\cdot n|_{\partial\Omega}=0, \quad \mbox{curl}H^{k+1}\times n|_{\partial\Omega}=0\\[2mm]
J^{k+1}\cdot n|_{\partial\Omega}=0, \quad \mbox{curl}J^{k+1}\times n|_{\partial\Omega}=-(\psi^{k+1}\cdot n)J^{k+1},\\[2mm]
(\phi^{k+1},\psi^{k+1},H^{k+1},J^{k+1},u^{k+1})\rightarrow (0,0,0,0,0),\quad \mbox{as}\quad |x|\rightarrow \infty, \  t>0.
\end{array}
\right.
\end{split}
\end{equation}
Through a similar computation as Subsection \ref{3.3}, we conclude that the sequences of solutions $(\phi^{k},\psi^{k},H^{k},J^{k},u^{k})$ meet the uniform a priori estimate \eqref{local-linear}.

Now we will prove the full sequence $(\phi^{k},\psi^{k},H^{k},J^{k},u^{k})$ converges  to a limit $(\phi,\psi,H,J,u)$
in some strong sense.

Set
\begin{align*}
&\overline{\phi}^{k+1}=\phi^{k+1}-\phi^{k},\quad \quad  \overline{ \psi}^{k+1} =\psi^{k+1}-\psi^{k},\quad \quad
\overline{H}^{k+1}=H^{k+1}-H^{k},\\
&\overline{J}^{k+1}=J^{k+1}-J^{k}, \quad \quad  \overline{u}^{k+1}=u^{k+1}-u^{k}.
\end{align*}
Then it follows from \eqref{inter-0} that
\begin{equation}\label{inter1}
\begin{split}
\left\{
\begin{array}{lll}
\overline{\phi}^{k+1}_t+u^{k}\cdot\nabla\overline{\phi}^{k+1}+\overline{u}^{k}\cdot\nabla \phi^{k}
+\frac{\gamma-1}{2}\left(\overline{\phi}^{k+1}\mbox{div}u^{k}+\phi^{k}\mbox{div}\overline{u}^{k}\right)=0,\\[4mm]
\partial_{t}\overline{\psi}^{k+1}_{j}+\partial_{j}u^{k}_i\overline{\psi}^{k+1}_i+\partial_{j}\overline{u}^{k}_{i}\psi^{k}_i
+u^{k}_{i}\partial_{j}\overline{\psi}^{k+1}_i+\overline{u}^{k}_i\partial_{j}\psi^{k}_i
+\partial_{j}\partial_{i}\overline{u}_i^{k}=0,\quad 1\leq i,j\leq 3,\\[4mm]
\overline{H}^{k+1}_{t}- \eta \Delta\overline{ H}^{k+1}=-u^{k}\cdot\nabla\overline{H}^{k+1}-\overline{u}^{k}\cdot\nabla H^{k}
+H^{k}\cdot\nabla \overline{u}^{k}+\overline{H}^{k+1}\cdot\nabla u^{k}
-H^{k}\mbox{div}\overline{u}^{k}-\overline{H}^{k+1}\mbox{div}u^{k},\\[2mm]
\begin{split}
\overline{J}^{k+1}_t-\eta\Delta \overline{J}^{k+1}=&-u^{k}\cdot\nabla \overline{J}^{k+1}-\overline{u}^{k}\cdot\nabla J^{k}
+\overline{J}^{k+1}\cdot\nabla u^{k}+J^{k}\cdot\nabla\overline{u}^{k}
+\eta J^{k}\mbox{div}\overline{\psi}^{k+1}\\[2mm]
&+\eta \overline{J}^{k+1}\mbox{div}\psi^{k+1}
+\eta J^{k}\left((\psi^{k+1})^2-(\psi^k)^2\right)+\eta \overline{J}^{k+1}(\psi^{k+1})^2\\[2mm]
&+2\eta \overline{\psi}^{k+1}\cdot \nabla J^{k}+2\eta\psi^{k+1}\cdot\nabla \overline{J}^{k+1}\equiv R_{4},\\[4mm]
\end{split}\\
\begin{split}
\overline{u}&^{k+1}_{t}+\alpha\mbox{curl}^2\overline{ u}^{k+1}-(2\alpha+\beta)\nabla\mbox{div}\overline{u}^{k+1}\\[2mm]
=&-u^{k}\cdot\nabla \overline{u}^{k}-\overline{u}^{k}\cdot\nabla u^{k-1}-\frac{A\gamma}{\gamma-1}\nabla\left((\phi^{k+1})^2-(\phi^{k})^2\right)
+\psi^{k+1} Q(\overline{u}^{k})+\overline{\psi}^{k+1}Q(u^{k-1})\\[2mm]
&+\overline{J}^{k+1}\cdot\nabla H^{k+1}+J^{k}\cdot\nabla\overline{ H}^{k+1}-
\overline{J}^{k+1}\cdot(\nabla H^{k+1})^{T}-J^{k}\cdot(\nabla\overline{ H}^{k+1})^{T}\equiv R_5,
\end{split}
\end{array}
\right.
\end{split}
\end{equation}
has been established with the following initial conditions
\begin{equation}
(\overline{\phi}^{k+1},\overline{\psi}^{k+1},\overline{H}^{k+1},\overline{J}^{k+1},\overline{u}^{k+1})|_{t=0}=(0,0,0,0,0)
\end{equation}
and boundary conditions
\begin{equation}\label{cond-uk}
\overline{u}^{k+1}\cdot n|_{\partial\Omega}=0,\quad \mbox{curl}\overline{u}^{k+1}\times n|_{\partial\Omega}=-K(x)\overline{u}^{k+1},
\end{equation}
\begin{equation}\label{cond-H}
\overline{H}^{k+1}\cdot n|_{\partial\Omega}=0, \quad \mbox{curl}\overline{H}^{k+1}\times n|_{\partial\Omega}=0,
\end{equation}
\begin{equation}\label{cond-Jk}
\overline{J}^{k+1}\cdot n|_{\partial\Omega}=0, \quad\mbox{curl}\overline{J}^{k+1}\times n|_{\partial\Omega}=-(\psi^{k+1}\cdot n)\overline{J}^{k+1}-(\overline{\psi}^{k+1}\cdot n)J^{k}.
\end{equation}

{\it Step 1:}  Estimate on $\|\overline{\phi}^{k+1}\|_{1}$. Multiplying $(\ref{inter1})_1$ by $\overline{\phi}^{k+1}$, integrating the result over $\Omega$, and applying boundary condition $u^k\cdot n|_{\partial\Omega}=0$ to have the following identity
\begin{equation*}
	-\int_{\Omega} u^k\cdot\nabla\overline{\phi}^{k+1}\cdot\overline{\phi}^{k+1}=\frac{1}{2}\int_{\Omega}\mbox{div}u^k|\overline{\phi}^{k+1}|^2,
	\end{equation*}
then it implies that
\begin{align}\label{phi-n-1}
\frac{1}{2}\frac{d}{dt}\|\overline{\phi}^{k+1}\|^2&=-\int_{\Omega}u^{k}\cdot\nabla\overline{ \phi}^{k+1}\cdot\overline{ \phi}^{k+1}
-\int_{\Omega}\left(\overline{u}^{k}\cdot\nabla \phi^{k}
+\frac{\gamma-1}{2}(\overline{\phi}^{k+1}\mbox{div}u^{k}+\phi^{k}\mbox{div}\overline{u}^{k})\right)\cdot\overline{ \phi}^{k+1}\nonumber\\[2mm]
&\leq  C\left(\|\nabla u^{k}\|_{L^{\infty}}\|\overline{\phi}^{k+1}\|^2+\|\overline{u}^{k}\|_{L^6}\|\nabla\phi^{k}\|_{L^3}\|\overline{\phi}^{k+1}\|
+\|\nabla\overline{u}^{k}\|\|\phi^{k}\|_{L^{\infty}}\|\overline{\phi}^{k+1}\|\right)\nonumber\\[2mm]
&\leq C_1^{k}(t)\|\overline{\phi}^{k+1}\|^2+\frac{\epsilon_1}{4} \|\nabla\overline{u}^{k}\|^2,
\end{align}
where
\begin{align}
C_1^{k}(t)=C\|\nabla u^{k}\|_{L^{\infty}}+C\epsilon_1^{-1}\left(\|\nabla\phi^{k}\|^2_{L^3}+\|\phi^{k}\|^2_{L^{\infty}}\right).\nonumber
\end{align}
Similarly, computing $\int_{\Omega}\partial_{x_i}(\ref{inter1})_1 \cdot\partial_{x_i}\overline{\phi}^{k+1}$, $i=1,2,3$, and utilizing
\begin{equation*}
	-\int_{\Omega} u_j^k\partial_{x_i}\partial_{x_j}\overline{\phi}^{k+1}\partial_{x_i}\overline{\phi}^{k+1}=\frac{1}{2}\int_{\Omega}\mbox{div} u^{k}|\nabla\overline{\phi}^{k+1}|^2 \leq \|\nabla u^{k}\|_{L^{\infty}}\|\nabla\overline{\phi}^{k+1}\|^2,
\end{equation*}
 it gives
\begin{align}\label{phi-n-2}
\frac{1}{2}\frac{d}{dt}\|\nabla\overline{\phi}^{k+1}\|^2
\leq& \|\nabla u^{k}\|_{L^{\infty}}\|\nabla\overline{\phi}^{k+1}\|^2+\|\nabla\overline{u}^{k}\|_{L^6}\|\nabla\phi^{k}\|_{L^3}\|\nabla\overline{\phi}^{k+1}\|
+\|\overline{u}^{k}\|_{L^6}\|\nabla^2\phi^k\|_{L^3}\|\nabla\overline{\phi}^{k+1}\|\nonumber\\[2mm]
&+\|\overline{\phi}^{k+1}\|_{L^6}\|\nabla^2u^{k}\|_{L^3}\|\nabla\overline{\phi}^{k+1}\|
+\|\nabla^2\overline{u}^{k}\|\|\phi^{k}\|_{L^{\infty}}\|\nabla\overline{\phi}^{k+1}\|\nonumber\\[2mm]
\leq &C_2^{k}(t)\|\nabla\overline{\phi}^{k+1}\|^2+\frac{\epsilon_1}{4} \|\nabla\overline{u}^{k}\|_1^2,
\end{align}
where \begin{align}
C_2^{k}(t)=C\left(\|\nabla u^{k}\|_{L^{\infty}}+\|\nabla^2u^{k}\|_{L^3}\right)+C\epsilon_1^{-1}\left(
\|\nabla\phi^{k}\|^2_{L^3}+\|\phi^{k}\|^2_{L^{\infty}}+\|\nabla^2\phi^{k}\|^2_{L^3}\right).\nonumber
\end{align}
Thus, the combination of \eqref{phi-n-1} and \eqref{phi-n-2} results in
\begin{align}\label{step1}
\frac{d}{dt}\|\overline{\phi}^{k+1}\|_1^2
\leq \left(C_1^{k}(t)+C_2^{k}(t)\right)\|\overline{\phi}^{k+1}\|_1^2+\epsilon_1 \|\nabla\overline{u}^{k}\|_1^2,
\end{align}
where
\begin{align}
\int_0^t(C_1^{k}(s)+C_2^{k}(s))ds\leq C+C\epsilon_1^{-1}t,\nonumber
\end{align}
thanks to the a priori estimate \eqref{local-linear}.

{\it Step 2:} Estimate on $\|\overline{\psi}^{k+1}\|$.
Multiplying $(\ref{inter1})_2$ by $\overline{\psi}^{k+1}$, $j=1,2,3$, integrating the result identity over $\Omega$, and summing them up yields
\begin{align}\label{psi^k+1}
\frac{1}{2}\frac{d}{dt}\|\overline{\psi}^{k+1}\|^2=&-\int_{\Omega}\left(\partial_{j}u^{k}_i\overline{\psi}^{k+1}_i
+\partial_{j}\overline{u}^{k}_{i}\psi^{k}_i
+\overline{u}^{k}_i\partial_{j}\psi^{k}_i\right) \overline{\psi}^{k+1}_j\nonumber\\[2mm]
&-\int_{\Omega}u^{k}_{i}\partial_{j}\overline{\psi}^{k+1}_i\overline{\psi}^{k+1}_j
-\int_{\Omega}\nabla \mbox{div}\overline{u}^{k} \cdot\overline{\psi}^{k+1}.
\end{align}
Integrating by parts with the boundary condition $u^{k}\cdot n|_{\partial\Omega}=0$, it holds
\begin{align}
-\int_{\Omega}u^{k}_{i}\partial_{j}\overline{\psi}^{k+1}_i\overline{\psi}^{k+1}_j
&=-\frac{1}{2}\int_{\Omega}u^{k}_{i}\partial_{i}\overline{\psi}^{k+1}_j\overline{\psi}^{k+1}_j
=\frac{1}{2}\int_{\Omega}\mbox{div}u^{k}|\overline{\psi}^{k+1}|^2\nonumber\\[2mm]
&\leq \|\nabla u^{k}\|_{L^{\infty}}\|\overline{\psi}^{k+1}\|^2\nonumber
\end{align}
due to $\partial_{j}\overline{\psi_i}^{k+1}=\partial_{i}\overline{\psi_j}^{k+1}$. Therefore, from (\ref{psi^k+1}), we get
\begin{align}\label{step2}
&\frac{d}{dt}\|\overline{\psi}^{k+1}\|^2
\leq
C^{k}_3(t)\|\overline{\psi}^{k+1}\|^2+\epsilon_1 \|\nabla \overline{u}^{k}\|_1^2,
\end{align}
where
\begin{align*}
C^{k}_3(t)=C\left(1+ \|\nabla u^{k}\|_{L^{\infty}}\right)+C\epsilon_1^{-1}\left(\|\psi^{k}\|^2_{L^{\infty}}+\|\nabla\psi^{k}\|^2_{L^3}\right)
\end{align*}
satisfies
 \begin{align}
 \int_{0}^{t}C^{k}_3(s)ds\leq C+C\epsilon_1^{-1}t. \nonumber
 \end{align}

{\it Step 3:} Estimate on $\|\overline{H}^{k+1}\|$. Multiplying $(\ref{inter1})_3$ by $\overline{H}^{k+1}$, integrating by parts with the boundary condition $\mbox{curl}\overline{H}^{k+1}\times n|_{\partial\Omega}=0$, one has
\begin{align}\label{Hk+1}
\frac{1}{2}\frac{d}{dt}\|\overline{H}^{k+1}\|^2+\eta\|\mbox{curl}\overline{H}^{k+1}\|^2
\leq\varepsilon_1^*\|\nabla\overline{H}^{k+1}\|^2+\epsilon_2 \|\nabla\overline{u}^{k}\|^2+\widetilde{ C^{k}_{4}}\|\overline{H}^{k+1}\|^2,
\end{align}
where
\begin{align}
\widetilde{C}^{k}_{4}=C\epsilon_2^{-1}\left(\|\nabla H^{k}\|^2_{L^3}+\|H^{k}\|^2_{L^{\infty}}\right)+C(\varepsilon_1^*)^{-1}\left(\|u^{k}\|^2_{L^{\infty}}+\|\nabla u^{k}\|^2_{L^3}\right).\nonumber
\end{align}
Because
\begin{align}
\|\nabla\overline{H}^{k+1}\|^2\leq C\left(\|\mbox{curl}\overline{H}^{k+1}\|^2 +\|\overline{H}^{k+1}\|^2\right),\nonumber
\end{align}
after choosing $\varepsilon_1^*$ small enough, \eqref{Hk+1} implies that there exists a constant $b_1>0$ such that
\begin{align}\label{step3}
\frac{d}{dt}\|\overline{H}^{k+1}\|^2+b_1\|\nabla\overline{H}^{k+1}\|^2
\leq C^{k}_{4}\|\overline{H}^{k+1}\|^2
+\epsilon_2 \|\nabla\overline{u}^{k}\|^2,
\end{align}
where
\begin{align}
C^{k}_{4}=C+\widetilde{C}^{k}_{4},\nonumber
\end{align}
with
  \begin{align}
 \int_{0}^{t}C^{k}_4(s)ds\leq  C+C\epsilon_2^{-1}t. \nonumber
 \end{align}

{\it Step 4:} Estimate on $\|\overline{J}^{k+1}\|$. Multiplying $(\ref{inter1})_4$ by $\overline{J}^{k+1}$, exploiting the boundary condition $\overline J^{k+1}\cdot n|_{\partial\Omega}=0$, then we have
\begin{align}\label{J^{k+1}}
&\frac{1}{2}\frac{d}{dt}\|\overline{J}^{k+1}\|^2+\eta\left(\|\mbox{curl}\overline{J}^{k+1}\|^2+\|\mbox{div}\overline{J}^{k+1}\|^2\right)
-\eta\int_{\partial\Omega}(\mbox{curl}\overline{J}^{k+1}\times n)\cdot \overline{J}^{k+1}\nonumber\\[2mm]
&=\int_{\Omega}R_4\cdot \overline{J}^{k+1}.
\end{align}
Using boundary condition \eqref{cond-Jk},  the boundary term becomes
\begin{align}\label{bdy-J^{k+1}}
-&\eta\int_{\partial\Omega}(\mbox{curl}\overline{J}^{k+1}\times n)\cdot \overline{J}^{k+1}
\nonumber\\[2mm]
=&
\eta\int_{\partial\Omega}(\psi^{k+1}\cdot n)\overline{J}^{k+1}\cdot \overline{J}^{k+1}+\eta\int_{\partial\Omega}(\overline{\psi}^{k+1}\cdot n)J^{k}\cdot \overline{J}^{k+1}\nonumber\\[2mm]
=&\eta\int_{\Omega}\mbox{div}\left(\overline{J}^{k+1}\cdot \overline{J}^{k+1}\psi^{k+1}\right)+\eta\int_{\Omega}\mbox{div}\left(J^{k}\cdot \overline{J}^{k+1}\overline{\psi}^{k+1}\right)\nonumber\\[2mm]
=&2\eta\int_{\Omega}\psi^{k+1}\cdot\nabla\overline{J}^{k+1}\cdot \overline{J}^{k+1}
+\eta\int_{\Omega}\mbox{div}\psi^{k+1}\left(\overline{J}^{k+1}\cdot \overline{J}^{k+1}\right)
\nonumber\\[2mm]
&+\eta\int_{\Omega}\overline{\psi}^{k+1}\cdot\nabla J^{k}\overline{J}^{k+1}
+\eta\int_{\Omega}\mbox{div}\overline{\psi}^{k+1}\left(J^{k}\cdot \overline{J}^{k+1}\right)
+\eta\int_{\Omega}\overline{\psi}^{k+1}\cdot\nabla \overline{J}^{k+1} J^{k}.
\end{align}
Remark that the first four  terms on the right-hand side of (\ref{bdy-J^{k+1}}) will be canceled out by the corresponding term of $\int_{\Omega}R_4\cdot \overline{J}^{k+1}$, then one obtains
\begin{align}
\frac{1}{2}&\frac{d}{dt}\|\overline{J}^{k+1}\|^2+\eta\left(\|\mbox{curl}\overline{J}^{k+1}\|^2
+\|\mbox{div}\overline{J}^{k+1}\|^2\right)\nonumber\\[2mm]
=&-\eta\int_{\Omega}\overline{\psi}^{k+1}\cdot\nabla \overline{J}^{k+1} J^{k}
+\int_{\Omega}\big(-u^{k}\cdot\nabla \overline{J}^{k+1}-\overline{u}^{k}\cdot\nabla J^{k}
+J^{k}\cdot\nabla\overline{u}^{k}+\overline{J}^{k+1}\cdot\nabla u^{k}
\big)\cdot\overline{J}^{k+1}\nonumber\\[2mm]
&+\int_{\Omega}\big(\eta J^{k}((\psi^{k+1})^2-(\psi^k)^2)+\eta \overline{J}^{k+1}(\psi^{k+1})^2
+\eta \overline{\psi}^{k+1}\cdot \nabla J^{k}
\big)\cdot  \overline{J}^{k+1}\nonumber\\[2mm]
\leq&
\|\overline{\psi}^{k+1}\| \|\nabla \overline{J}^{k+1}\| \|J^{k}\|_{L^{\infty}}
+\|u^{k}\|_{L^{\infty}}\|\nabla \overline{J}^{k+1}\|\|\overline{J}^{k+1}\|
+\|\overline{u}^{k}\|_{L^6}\|\nabla J^{k}\|_{L^3}\|\overline{J}^{k+1}\|\nonumber\\[2mm]
&+\|J^{k}\|_{L^{\infty}}\|\nabla\overline{ u}^{k}\| \|\overline{J}^{k+1}\|
+\|\overline{J}^{k+1}\|_{L^6}\|\nabla u^{k}\|_{L^3} \|\overline{J}^{k+1}\|
\nonumber\\[2mm]
&+\|J^k\|_{L^6}\|\overline{\psi}^{k+1}\|\|\overline{J}^{k+1}\|_{L^6}\left(\|\psi^{k+1}\|_{L^6}+\|\psi^{k}\|_{L^6}\right)\nonumber\\[2mm]
&+ \|\overline{J}^{k+1}\|_{L^6}\|\psi^{k+1}\|^2_{L^6}\|\overline{J}^{k+1}\|_{L^2}
+ \|\overline{\psi}^{k+1}\|_{L^2} \|\nabla J^{k}\|_{L^3} \|\overline{J}^{k+1}\|_{L^6}. \nonumber
\end{align}
It indicates that
\begin{align}\label{step4-0}
&\frac{1}{2}\frac{d}{dt}\|\overline{J}^{k+1}\|^2+\eta\left(\|\mbox{curl}\overline{J}^{k+1}\|^2
+\|\mbox{div}\overline{J}^{k+1}\|^2\right)\nonumber\\[2mm]
&\leq\widetilde{C}^{k}_{6}(t) \|\overline{J}^{k+1}\|^2
+C^{k}_{7}(t)\|\overline{\psi}^{k+1}\|^2+ \varepsilon_2^*\|\nabla\overline{J}^{k+1} \|^2
+\epsilon_3\|\nabla\overline{u}^{k}\|^2,
\end{align}
where
\begin{align}
\widetilde{ C}^{k}_{6}(t)&=C{(\varepsilon_2^*)}^{-1}\Big(\|u^k\|^2_{L^{\infty}}+\|\nabla u^{k}\|^2_{L^3}
+\|\psi^{k+1}\|^4_{L^{6}}\Big)+C
{\epsilon_3}^{-1}\left(\|\nabla J^{k}\|^2_{L^3}+\|J^{k}\|^2_{L^{\infty}}\right),\nonumber\\[2mm]
C^{k}_{7}(t)&=C(\varepsilon_2^*)^{-1}\left(\|J^{k}\|^2_{L^{\infty}}
+\|J^k\|^2_{L^6}(\|\psi^{k+1}\|^2_{L^6}+\|\psi^{k}\|^2_{L^6})+\|\nabla J^{k}\|^2_{L^{3}}\right).\nonumber
 \end{align}
By using Lemma \ref{lem-div-curl} to $\overline{J}^{k+1}$ and choosing $\varepsilon^*_2$ small,  \eqref{step4-0}  yields that
\begin{align}\label{step4}
\frac{d}{dt}\|\overline{J}^{k+1}\|^2+b_{2}\|\nabla\overline{J}^{k+1}\|^2
\leq C^{k}_{6}(t) \|\overline{J}^{k+1}\|^2
+C^{k}_{7}(t)\|\overline{\psi}^{k+1}\|^2
+\epsilon_3\|\nabla\overline{u}^{k}\|^2,
 \end{align}
where
 \begin{align}
C^{k}_6(t)&=\widetilde{C}^{k}_6(t)+C,\nonumber
 \end{align}
satisfies
 \begin{align}
 \int_{0}^{t}(C^{k}_6(s)+C^{k}_7(s))ds\leq  C+C\epsilon_3^{-1}t. \nonumber
 \end{align}

{\it  Step 5:} Estimate on $\|\overline{u}^{k+1}\|_1$. Multiplying $(\ref{inter1})_5$ by $\overline{u}^{k+1}$, integrating the result over $\Omega$, and integrating by parts with the boundary condition (\ref{cond-uk}), one has
\begin{align}
\frac{1}{2}&\frac{d}{dt}\|\overline{u}^{k+1}\|^2+\alpha\|\mbox{curl}\overline{u}^{k+1}\|^2
+(2\alpha+\beta)\|\mbox{div}\overline{u}^{k+1}\|^2
+\alpha\int_{\partial\Omega}K(x)|\overline{u}^{k+1}|^2\nonumber\\[2mm]
=&\int_{\Omega}\left(-u^{k}\cdot\nabla \overline{u}^{k}
-\overline{u}^{k}\cdot\nabla u^{k-1}
\right)\cdot \overline{u}^{k+1}
+\int_{\Omega}
\frac{A\gamma}{\gamma-1}\left((\phi^{k+1})^2-(\phi^{k})^2\right)\mbox{div}\overline{u}^{k+1}\nonumber\\[2mm]
&+\int_{\Omega}\left(\psi^{k+1} Q(\overline{u}^{k})+\overline{\psi}^{k+1}Q(u^{k-1})\right)\cdot \overline{u}^{k+1}\nonumber\\[2mm]
&+\int_{\Omega}\left(\overline{J}^{k+1}\cdot\nabla H^{k+1}+J^{k}\cdot\nabla\overline{ H}^{k+1}-
\overline{J}^{k+1}\cdot(\nabla H^{k+1})^{T}-J^{k}\cdot(\nabla\overline{ H}^{k+1})^{T}\right)\cdot \overline{u}^{k+1}\nonumber\\[2mm]
\leq& \left(\|u^{k}\|_{L^{\infty}}\|\nabla \overline{u}^{k}\|
+\|\overline{u}^{k}\|_{L^6}\|\nabla u^{k-1}\|_{L^3}
\right) \| \overline{u}^{k+1}\|
+\|\overline{\phi}^{k+1}\|\|\phi^{k+1}+\phi^{k}\|_{L^{\infty}}\|\mbox{div}\overline{u}^{k+1}\|\nonumber \\[2mm]
&+\|\psi^{k+1}\|_{L^{\infty}}\|\nabla\overline{u}^{k}\| \|\overline{u}^{k+1}\|+\|\overline{\psi}^{k+1}\| \|\nabla u^{k-1}\|_{L^{3}} \|\overline{u}^{k+1}\|_{L^6}\nonumber \\[2mm]
&+ \|\overline{J}^{k+1}\| \|\nabla H^{k+1}\|_{L^{3}}\|\overline{u}^{k+1}\|_{L^6}
+\|J^{k}\|_{L^{\infty}}\|\nabla\overline{H}^{k+1}\| \|\overline{u}^{k+1}\|.\nonumber
\end{align}
Therefore, one obtains
\begin{align}\label{u^{k+1}}
\frac{1}{2}&\frac{d}{dt}\|\overline{u}^{k+1}\|^2+\alpha\|\mbox{curl}\overline{u}^{k+1}\|^2
+(2\alpha+\beta)\|\mbox{div}\overline{u}^{k+1}\|^2
+\alpha\int_{\partial\Omega}K(x)|\overline{u}^{k+1}|^2\nonumber\\[2mm]
\leq&\widetilde{ C}^{k}_{8}(t)\|\overline{u}^{k+1}\|^2+ C^{k}_9(t)\|\overline{\phi}^{k+1}\|^2+C^k_{10}(t)\|\overline{\psi}^{k+1}\|^2 +C^{k}_{11}(t)\|\overline{J}^{k+1}\|^2\nonumber\\[2mm]
&+\varepsilon_3^* \|\nabla\overline{u}^{k+1}\|^2+\varepsilon_4^*\|\nabla\overline{H}^{k+1}\|^2+\epsilon_4 \|\nabla\overline{u}^{k}\|^2,
\end{align}
where
\begin{align}
\widetilde{C}^{k}_8(t)=C\epsilon_4^{-1}\left(\|u^{k}\|^2_{L^{\infty}}+\|\nabla u^{k-1}\|^2_{L^3}+\|\psi^{k+1}\|^2_{L^{\infty}}
\right)+C(\varepsilon_4^*)^{-1}\|J^{k}\|^2_{L^{\infty}},\quad \varepsilon_4^*=\frac{b_1}{2}\nonumber
\end{align}
and
\begin{align}
C^{k}_9(t)=C(\varepsilon_3^*)^{-1}\|\phi^{k+1}+\phi^{k}\|^2_{L^{\infty}},\quad C^{k}_{10}(t)=C(\varepsilon_3^*)^{-1}\|\nabla u^{k-1}\|^2_{L^3},\quad C^{k}_{11}(t)=C(\varepsilon_3^*)^{-1}\|\nabla H^{k+1}\|^2_{L^{3}}. \nonumber
\end{align}
Therefore, applying (\ref{curl-div}) of Lemma \ref{lem-div-curl} to $\overline{u}^{k+1}$, and choosing $\varepsilon_3^*$ small enough in \eqref{u^{k+1}}, there exists a constant $b_3>0$ such that
\begin{align}\label{step5}
\frac{d}{dt}&\|\overline{u}^{k+1}\|^2+b_{3}\|\nabla\overline{u}^{k+1}\|^2
\nonumber\\[2mm]
\leq& C^{k}_{8}(t)\|\overline{u}^{k+1}\|^2+ C^{k}_9(t)\|\overline{\phi}^{k+1}\|^2+C^k_{10}(t)\|\overline{\psi}^{k+1}\|^2 +C^{k}_{11}(t)\|\overline{J}^{k+1}\|^2\nonumber\\[2mm]
&+\epsilon_4 \|\nabla\overline{u}^{k}\|^2+\varepsilon_4^*\|\nabla\overline{H}^{k+1}\|^2,
\end{align}
where
\begin{align}
\widetilde{C}^{k}_8(t)=C^{k}_8(t)+C.\nonumber
\end{align}

 Furthermore, we multiply $(\ref{inter1})_5$ by $-\nabla\mbox{div}\overline{u}^{k+1}$ and integrate the  resulting identity over $\Omega$ to derive
  \begin{align}\label{section5-1}
&- \int_{\Omega} \overline{u}^{k+1}_{t}\cdot \nabla\mbox{div}\overline{u}^{k+1}+(2\alpha+\beta)\|\nabla\mbox{div}\overline{u}^{k+1}\|^2-\alpha\int_{\Omega}\mbox{curl}^2\overline{ u}^{k+1}\cdot \nabla\mbox{div}\overline{u}^{k+1}\nonumber\\[2mm]
&=-\int_{\Omega}R_5\cdot \nabla\mbox{div}\overline{u}^{k+1}.
  \end{align}
   Notice that the boundary condition $\overline{u}^{k+1}\cdot n|_{\partial\Omega}=0$ implies
  \begin{align}\label{u_t}
-\int_{\Omega}\overline{u}_t^{k+1}\cdot \nabla\mbox{div} \overline{u}^{k+1}
=\frac{1}{2}\frac{d}{dt}\int_{\Omega}|\mbox{div}\overline{u}^{k+1}|^2.
\end{align}
Next, on the one hand, since $-K(x)\overline{u}^{k+1}\cdot n=(\mbox{curl}\overline{u}^{k+1}\times n)\cdot n=0$ on $\partial\Omega$, it is easy to see
 \begin{equation}
 K(x)\overline{u}^{k+1}=-(K(x)\overline{u}^{k+1}\times n)\times n\doteq (K(x)\overline{u}^{k+1})^{\perp}\times n \quad \mbox{on}\quad  \partial\Omega,\nonumber
 \end{equation}
which implies that
 \begin{align}\label{before}
\int_{\Omega}\mbox{curl}^2\overline{u}^{k+1}\cdot \nabla\mbox{div} \overline{u}^{k+1}
&=-\int_{\partial\Omega}(\mbox{curl}\overline{u}^{k+1}\times n)\cdot \nabla\mbox{div} \overline{u}^{k+1}\nonumber\\[2mm]
&=\int_{\partial\Omega} K(x)\overline{u}^{k+1}\cdot \nabla\mbox{div} \overline{u}^{k+1}\nonumber\\[2mm]
&=\int_{\partial\Omega} ((K(x)\overline{u}^{k+1})^{\perp}\times n)\cdot \nabla\mbox{div} \overline{u}^{k+1}\nonumber\\[2mm]
&=\int_{\partial\Omega} (\nabla\mbox{div} \overline{u}^{k+1}\times (K(x)\overline{u}^{k+1})^{\perp} )\cdot n \nonumber\\[2mm]
&=\int_{\Omega} \mbox{div}\big(\nabla\mbox{div} \overline{u}^{k+1}\times (K(x)\overline{u}^{k+1})^{\perp} \big) \nonumber\\[2mm]
&=-\int_{\Omega} \nabla\times (K(x)\overline{u}^{k+1})^{\perp} \cdot \nabla\mbox{div} \overline{u}^{k+1}\nonumber\\[2mm]
&\leq\varepsilon_5\|\nabla\mbox{div} \overline{u}^{k+1}\|^2 + C \|\overline{u}^{k+1}\|_1^2,
\end{align}
where we have used the identity $\mbox{div}(\nabla f\times v)=-\nabla\times v\cdot \nabla f$ in the last equality.
On the other hand, the right-hand side of \eqref{section5-1} can  be controlled as
\begin{align}\label{R5}
\int_{\Omega}&R_5\cdot \nabla\mbox{div} \overline{u}^{k+1}\nonumber\\[2mm]
=&\int_{\Omega}\left(-u^{k}\cdot\nabla \overline{u}^{k}
-\overline{u}^{k}\cdot\nabla u^{k-1}
\right)\cdot \nabla\mbox{div}\overline{u}^{k+1}
+\int_{\Omega}
\frac{A\gamma}{\gamma-1}\nabla\left((\phi^{k+1})^2-(\phi^{k})^2\right)\cdot\nabla\mbox{div}(\overline{u}^{k+1})\nonumber\\[2mm]
&+\int_{\Omega}\left(\psi^{k+1} Q(\overline{u}^{k})
+\overline{\psi}^{k+1}Q(u^{k-1})\right)\cdot\nabla\mbox{div} \overline{u}^{k+1}\nonumber\\[2mm]
&+\int_{\Omega}\left(\overline{J}^{k+1}\cdot\nabla H^{k+1}+J^{k}\cdot\nabla\overline{ H}^{k+1}-
\overline{J}^{k+1}\cdot(\nabla H^{k+1})^{T}-J^{k}\cdot(\nabla\overline{ H}^{k+1})^{T}\right)
\cdot \nabla\mbox{div}\overline{u}^{k+1}\nonumber\\[2mm]
\leq & \frac{1}{8}(2\alpha+\beta)\|\nabla\mbox{div}\overline{u}^{k+1}\|^2+C(
\|u^{k}\|^2_{L^{\infty}}\|\nabla \overline{u}^{k}\|^2_{L^2}+\|\overline{u}^{k}\|^2_{L^6}\|\nabla u^{k-1}\|^2_{L^3})\nonumber\\[2mm]
&+C(\|\nabla\overline{\phi} ^{k+1}\|^2 \|\phi^{k+1}+\phi^{k}\|^2_{L^{\infty}}+ \|\overline{\phi} ^{k+1}\|^2_{L^6} \|\nabla(\phi^{k+1}+\phi^{k})\|^2_{L^{3}}+\|\psi^{k+1}\|^2_{L^{\infty}} \|\nabla\overline{u}^{k}\|^2)\nonumber\\[2mm]
&+C(\|\overline{\psi}^{k+1}\|^2\|\nabla u^{k-1}\|^2_{L^{\infty}}+
 \|\overline{J}^{k+1}\|^2\|\nabla H^{k+1}\|^2_{L^{\infty}}+\|J^{k}\|^2_{L^{\infty}}\|\nabla\overline{ H}^{k+1}\|^2).
\end{align}
Putting \eqref{u_t}-\eqref{R5} into \eqref{section5-1}, and choosing $\varepsilon_5=\frac{1}{8\alpha}(2\alpha+\beta)$ in \eqref{before}, it leads to
\begin{align}\label{divu}
\frac{1}{2}&\frac{d}{dt}\|\mbox{div}\overline{u}^{k+1}\|^2
+\frac{6}{8}(2\alpha+\beta)\|\nabla\mbox{div}\overline{u}^{k+1}\|^2\nonumber\\[2mm]
 \leq&  C \|\overline{u}^{k+1}\|_1^2+C^k_{12}(t)\|\nabla\overline{\phi} ^{k+1}\|^2+
C^{k}_{13}(t)\|\overline{\psi}^{k+1}\|^2+C^{k}_{14}(t)\|\overline{J}^{k+1}\|^2 \nonumber \\[2mm]
&+ C\left(\|u^k\|^2_{L^{\infty}}+\|\nabla u^{k-1}\|^2_{L^3}+\|\psi^{k+1}\|^2_{L^{\infty}}\right)
\|\nabla\overline{u}^{k} \|^2+C\|J^{k}\|^2_{L^{\infty}}\|\nabla \overline{H}^{k+1}\|^2,
\end{align}
where
\begin{align*}
C^{k}_{12}(t)=\|\phi^{k+1}+\phi^{k}\|^2_{L^{\infty}}+\|\nabla\phi^{k+1}+\nabla\phi^{k}\|^2_{L^3},
\end{align*}
\begin{align*}
C^{k}_{13}(t)=\|\nabla u^{k-1}\|^2_{L^{\infty}},\quad C^{k}_{14}(t)=\|\nabla H^{k+1}\|^2_{L^{\infty}}.
\end{align*}

Similarly, multiplying $(\ref{inter1})_5$  by $\mbox{curl}^2\overline{u}^{k+1}$, integrating the result over $\Omega$, it holds
  \begin{align}
&\int_{\Omega} \overline{u}^{k+1}_{t}\cdot \mbox{curl}^2\overline{u}^{k+1}+\alpha\|\mbox{curl}^2\overline{ u}^{k+1}\|^2-(2\alpha+\beta)\int_{\Omega}\nabla\mbox{div}\overline{u}^{k+1}\cdot \mbox{curl}^2\overline{u}^{k+1}\nonumber\\[2mm]
&=\int_{\Omega}R_5\cdot \mbox{curl}^2\overline{u}^{k+1}.\nonumber
  \end{align}
 The first term on the left-hand side of the above equation can be expressed as
\begin{align}
\int_{\Omega}\overline{u}_t^{k+1}\cdot \mbox{curl}^2 \overline{u}^{k+1}&=\int_{\Omega}\mbox{curl}\overline{u}_t^{k+1}\cdot \mbox{curl} \overline{u}^{k+1}-\int_{\partial\Omega}\overline{u}_t^{k+1}\cdot (\mbox{curl} \overline{u}^{k+1}\times n)
\nonumber\\[2mm]
&=\frac{1}{2}\frac{d}{dt}\int_{\Omega}|\mbox{curl}\overline{u}^{k+1}|^2+\int_{\partial\Omega}\overline{u}_t^{k+1}\cdot K(x)\cdot \overline{u}^{k+1}\nonumber\\[2mm]
&=\frac{1}{2}\frac{d}{dt}\int_{\Omega}|\mbox{curl}\overline{u}^{k+1}|^2+\frac{1}{2}\frac{d}{dt}\int_{\partial\Omega}\overline{u}^{k+1}\cdot K(x)\cdot \overline{u}^{k+1}.\nonumber
\end{align}
Therefore, employing (\ref{before}) with $\varepsilon_5=\frac{1}{8}(2\alpha+\beta)$ produces
\begin{align}\label{curlu}
&\frac{1}{2}\frac{d}{dt}\|\mbox{curl}\overline{u}^{k+1}\|^2+\frac{1}{2}\frac{d}{dt}\int_{\partial\Omega}\overline{u}^{k+1}\cdot K(x)\cdot \overline{u}^{k+1} +\alpha \|\mbox{curl}^2\overline{u}^{k+1}\|^2\nonumber\\[2mm]
&\leq  \frac{1}{8}(2\alpha+\beta) \|\nabla\mbox{div} \overline{u}^{k+1}\|^2 + C \|\overline{u}^{k+1}\|_{1}^2
+\int_{\Omega} R_5\cdot\mbox{curl}^2\overline{u}^{k+1}.
\end{align}
We can estimate the term $\int_{\Omega}R_5\cdot\mbox{curl}^2\overline{u}^{k+1}$ by a similar argument as \eqref{R5}.

Therefore, combining \eqref{divu} with \eqref{curlu}, it indicates
\begin{align}
\frac{d}{dt}&\left(\|\mbox{div}\overline{u}^{k+1}\|^2+\|\mbox{curl}\overline{u}^{k+1}\|^2\right)+\frac{d}{dt}\int_{\partial\Omega}\overline{u}^{k+1}\cdot K(x)\cdot \overline{u}^{k+1} +C\left(\|\nabla\mbox{div}\overline{u}^{k+1}\|^2+\|\mbox{curl}^2\overline{u}^{k+1}\|^2\right)\nonumber\\[2mm]
\leq&  C \|\overline{u}^{k+1}\|_1^2+C^k_{12}(t)\|\nabla\overline{\phi} ^{k+1}\|^2+
C^{k}_{13}(t)\|\overline{\psi}^{k+1}\|^2+C^{k}_{14}(t)\|\overline{J}^{k+1}\|^2 \nonumber \\[2mm]
&+ C\left(\|u^k\|^2_{L^{\infty}}+\|\nabla u^{k-1}\|^2_{L^3}+\|\psi^{k+1}\|^2_{L^{\infty}}\right)
\|\nabla\overline{u}^{k} \|^2+C\|J^{k}\|^2_{L^{\infty}}\|\nabla \overline{H}^{k+1}\|^2,\nonumber
\end{align}
which implies that there exists a constant $b_4>0$ such that
\begin{align}\label{step6}
\frac{d}{dt}&\left(\|\mbox{div}\overline{u}^{k+1}\|^2+\|\mbox{curl}\overline{u}^{k+1}\|^2\right)+\frac{d}{dt}\int_{\partial\Omega}\overline{u}^{k+1}\cdot K(x)\cdot \overline{u}^{k+1} +b_4\|\nabla^2\overline{u}^{k+1}\|^2\nonumber\\[2mm]
\leq&  C \|\overline{u}^{k+1}\|_1^2+C^k_{12}(t)\|\nabla\overline{\phi} ^{k+1}\|^2+
C^{k}_{13}(t)\|\overline{\psi}^{k+1}\|^2+C^{k}_{14}(t)\|\overline{J}^{k+1}\|^2 \nonumber \\[2mm]
&+ C\left(\|u^k\|^2_{L^{\infty}}+\|\nabla u^{k-1}\|^2_{L^3}+\|\psi^{k+1}\|^2_{L^{\infty}}\right)
\|\nabla\overline{u}^{k} \|^2+C\|J^{k}\|^2_{L^{\infty}}\|\nabla \overline{H}^{k+1}\|^2,
\end{align}
 thanks to following elliptic estimate
		\begin{align}
	\|\nabla^2\overline u^{k+1}\|^2
	\leq C\left(\|\mbox{curl}^2 \overline u^{k+1}\|^2+\|\nabla\mbox{div}\overline u^{k+1}\|^2+\|\nabla\overline u^{k+1}\|^2+\|\overline u^{k+1}\|^2\right).\nonumber
	\end{align}

{\it  Step 6:} The full sequence $(\phi^{k},\psi^{k},H^{k},J^{k},u^{k})$ converges to a limit $(\phi,\psi,H,J,u)$ in some strong sense. Denote
\begin{align}
\Gamma ^{k+1}=& \|\overline{\phi}^{k+1}\|^2_1+\|\overline{\psi}^{k+1}\|^2
+C_1\|\overline{u}^{k+1}\|^2
+ \left(\|\mbox{div}\overline{u}^{k+1}\|^2+ \|\mbox{curl}\overline{u}^{k+1}\|^2\right)\nonumber\\[2mm]
&+C_2\|\overline{H}^{k+1}\|^2+\|\overline{J}^{k+1}\|^2,\nonumber
\end{align}
for some constants $C_1$ and $C_2$ ($C_1<C_2$) to be determined later.
Then it follows from inequalities \eqref{step1}, \eqref{step2}, \eqref{step3}, \eqref{step4}, \eqref{step5} and \eqref{step6} that
\begin{align}\label{finnally}
\frac{d}{dt}&\Gamma ^{k+1}+\frac{d}{dt}\int_{\partial\Omega}\overline{u}^{k+1}\cdot K(x)\cdot \overline{u}^{k+1}+\left(C_1b_3\|\nabla \overline{u}^{k+1}\|^2+\|\nabla^2 \overline{u}^{k+1}\|^2+\frac{C_2b_1}{2}\|\nabla \overline{H}^{k+1}\|^2
+b_2\|\nabla \overline{J}^{k+1}\|^2\right)\nonumber\\[2mm]
\leq& \mathcal{C}^{k}(t)\Gamma ^{k+1}+ 2\epsilon \|\nabla^2\overline{u}^{k}\|^2+(C_1+C_2+3)\epsilon \|\nabla \overline{u}^{k}\|^2\nonumber\\[2mm]
&+ C\left(\|u^k\|^2_{L^{\infty}}+\|\nabla u^{k-1}\|^2_{L^3}+\|\psi^{k+1}\|^2_{L^{\infty}}\right)
\|\nabla\overline{u}^{k} \|^2+C\|J^{k}\|^2_{L^{\infty}}\|\nabla \overline{H}^{k+1}\|^2\nonumber\\[2mm]
\leq& \mathcal{C}^{k}(t)\Gamma ^{k+1}+ 2\epsilon \|\nabla^2\overline{u}^{k}\|^2+(C_1+C_2+3)\epsilon \|\nabla \overline{u}^{k}\|^2+ C\|\nabla\overline{u}^{k} \|^2+C\|\nabla \overline{H}^{k+1}\|^2
\end{align}
where
	 $$\epsilon=\max\{\epsilon_1,\epsilon_2,\epsilon_3,\epsilon_4\},\quad\mathcal{C}^{k}(s)=C_1\mathcal{C}^{k}_{8}(t)+C_2\mathcal{C}^{k}_{4}(t)+\sum_{i=1,i\ne4,8}^{14}\mathcal{C}^{k}_{i}(t)>0.$$
and $C$ is independent of $\epsilon_i, (i=1,2,3,4)$, $C_1$ and $C_2$.
Let
\begin{align}
G^{k+1}= \|\overline{\phi}^{k+1}\|^2_1+\|\overline{\psi}^{k+1}\|^2
+C_1\|\overline{u}^{k+1}\|^2+\|\nabla\overline{u}^{k+1}\|^2
+C_2\|\overline{H}^{k+1}\|^2+\|\overline{J}^{k+1}\|^2.\nonumber
\end{align}
 Integrating (\ref{finnally}) over $s\in (0,t)$, applying  the trace theorem  with $C_1$ large enough, it yields that
\begin{align}
	 G^{k+1}&(t)+\int_0^t \left(C_1b_3\|\nabla \overline{u}^{k+1}\|^2+\|\nabla^2 \overline{u}^{k+1}\|^2+\frac{C_2b_1}{2}\|\nabla \overline{H}^{k+1}\|^2
	+b_2\|\nabla \overline{J}^{k+1}\|^2\right) ds\nonumber\\[2mm]
	\leq &\int_0^t\mathcal{C}^{k}G^{k+1}ds+\Big\{C\int_0^t\left(\epsilon \|\nabla^2\overline{u}^{k}\|^2+C_2\epsilon \|\nabla \overline{u}^{k}\|^2
		\right)ds\nonumber\\[2mm]
	&+C\int_0^t\left(
	\|\nabla\overline{u}^{k} \|^2+\|\nabla \overline{H}^{k+1}\|^2\right)ds+G^{k+1}(0)\Big\} \nonumber
\end{align} 	
where  	
 	\begin{align}
 		\int_{0}^{t}\mathcal{C}^{k}(s)ds\leq C_{*}+C_{*}\epsilon_*^{-1}t,\quad\quad\epsilon_*=\min\{\epsilon_1,\epsilon_2,\epsilon_3,\epsilon_4\} .\nonumber
 	\end{align}
 	
 Now we choose $C_2$ large enough to cancel the term $C\int_0^t\|\nabla \overline{H}^{k+1}\|^2ds$ on the right-hand sides of the above inequality, and then employ Gronwall's inequality to \eqref{finnally}, we have
\begin{align}
&G^{k+1}(s)+\int_0^t \left(C_1b_3\|\nabla \overline{u}^{k+1}\|^2+\|\nabla^2 \overline{u}^{k+1}\|^2+C_2b_1\|\nabla \overline{H}^{k+1}\|^2
+b_2\|\nabla \overline{J}^{k+1}\|^2\right) ds\nonumber\\[2mm]
&\leq \Big\{C\int_0^t \left(\epsilon \|\nabla^2\overline{u}^{k}\|^2+C_2\epsilon \|\nabla \overline{u}^{k}\|^2
\right)ds+C\int_0^t
\|\nabla\overline{u}^{k} \|^2ds+C\Big\} \cdot\exp\left(C_{*}+C_{*}\epsilon_*^{-1}t\right).\nonumber
  \end{align}
 We can firstly pick up  $\epsilon>0$ small  and then $T_{*}\in(0,\min(1,T^{**}))$ small enough such that
\begin{align}
&C\exp(C_{*})\leq \frac{C_1b_3}{8}, \quad C\epsilon\exp(C_{*})\leq  \frac{C_1b_3}{8};\nonumber\\
&C\epsilon\exp(C_{*})\leq \frac{1}{8}\min\{1,b_3\},\nonumber\\
& C\exp(C_{*}\epsilon_*^{-1}T_{*})\leq \frac{1}{8}\min\{1,b_3\}.\nonumber
\end{align}
Then it is easy to see that
\begin{align}
&\sum _{k=1}^{\infty}\sup_{0\leq t\leq T_*}\left(\|\overline{\phi}^{k+1}\|^2_1+\|\overline{\psi}^{k+1}\|^2+\|\overline{u}^{k+1}\|^2_1
+\|\overline{H}^{k+1}\|^2+\|\overline{J}^{k+1}\|^2\right)\nonumber\\[2mm]
&+\sum _{k=1}^{\infty}\int_0^{T_*} \left(\|\nabla \overline{u}^{k+1}\|^2+\|\nabla^2 \overline{u}^{k+1}\|^2+\|\nabla \overline{H}^{k+1}\|^2
+\|\nabla \overline{J}^{k+1}\|^2\right) ds\leq C.\nonumber
\end{align}
Moreover, the interpolation inequality is used to show that
\begin{align}
\lim_{k\rightarrow \infty}\|\nabla \overline{\psi}^{k+1}\|_{L^2}\leq C\lim_{k\rightarrow \infty}
\| \overline{\psi}^{k+1}\|^{\frac{1}{2}}_{L^2}\| \nabla^2\overline{\psi}^{k+1}\|^{\frac{1}{2}}_{L^2}\leq C\lim_{k\rightarrow \infty}
\| \overline{\psi}^{k+1}\|^{\frac{1}{2}}_{L^2} =0.\nonumber
\end{align}
Thus, the whole sequence $(\phi^{k},\psi^{k},H^{k},J^{k},u^{k})$ converges
to a limit $(\phi,\psi,H,J,u)$ in the following strong sense:
\begin{align}\label{solu}
&\phi^{k}\rightarrow \phi \quad \mbox{in}\quad L^{\infty}(0,T_*;H^1(\Omega)),\nonumber\\[2mm]
&\psi^{k}\rightarrow \psi \quad \mbox{in}\quad L^{\infty}(0,T_*;D^1(\Omega)),\nonumber\\[2mm]
&(H^{k},J^{k})\rightarrow (H,J) \quad \mbox{in}\quad L^{\infty}(0,T_*;L^2(\Omega))\cap L^{2}(0,T_*;D^1(\Omega)),\nonumber\\[2mm]
&u^{k}\rightarrow u \quad \mbox{in}\quad  L^{\infty}(0,T_*;H^1(\Omega))\cap L^{2}(0,T_*;D^1\cap D^2(\Omega)).
\end{align}
Thanks to (\ref{solu}), $(\phi, \psi, H, J,u )$ is a weak solution of
problem (\ref{J})-(\ref{initial-2}) in the distribution sense. Furthermore, due to the a priori estimate \eqref{local-linear}  and the lower-continuity of norm for weak or weak$^{*}$convergence, $(\phi, \psi, H, J,u )$ also satisfies the estimate \eqref{local-linear}.

{\it Step 7:} Finally, we shall prove the uniqueness of the strong solution to reformulated problem $(\ref{J})$-$(\ref{initial-2})$. Assume that $(\phi_1, \psi_1, H_1, J_1,u_1)$ and $(\phi_2, \psi_2, H_2, J_2,u_2)$ are two strong solutions to initial-boundary value problem $(\ref{J})$-$(\ref{initial-2})$, and satisfy the uniform estimates in \eqref{local-linear}. Define
$$	\overline\phi=\phi_1-\phi_2,\quad 	\overline\psi=\psi_1-\psi_2,\quad 	\overline{H}=H_1-H_2,\quad	 \overline{J}=J_1-J_2,\quad	\overline{u} =u_1-u_2.$$
According to the initial-boundary value problem $(\ref{J})$-$(\ref{initial-2})$, $(\overline\phi,\overline\psi,\overline H,\overline{J},\overline{u})$ should satisfy the following system
\begin{equation*}
	\begin{split}
		\left\{
		\begin{array}{lll}
			\overline{\phi}_t+u_{1}\cdot\nabla\overline{\phi}+\overline{u}\cdot\nabla \phi_2
			 +\frac{\gamma-1}{2}\left(\overline{\phi}\mbox{div}u_1+\phi_2\mbox{div}\overline{u}\right)=0,\\[4mm]
			 \partial_{t}\overline{\psi}_{j}+\partial_{j}u_{1i}\overline{\psi}_i+\partial_{j}\overline{u}_{i}\psi_{2i}
			+u_{1i}\partial_{j}\overline{\psi}_i+\overline{u}_i\partial_{j}\psi_{2i}
			+\partial_{j}\partial_{i}\overline{u}_i=0,\quad 1\leq i,j\leq 3,\\[4mm]
			\overline{H}_{t}- \eta \Delta\overline{ H}=-u_{1}\cdot\nabla\overline{H}-\overline{u}\cdot\nabla H_2
			+H_2\cdot\nabla \overline{u}+\overline{H}\cdot\nabla u_1
			-H_2\mbox{div}\overline{u}-\overline{H}\mbox{div}u_1,\\[2mm]
			\begin{split}
			\overline{J}_t-\eta\Delta \overline{J}=&-u_1\cdot\nabla \overline{J}-\overline{u}\cdot\nabla J_{2}
			+\overline{J}\cdot\nabla u_1+J_2\cdot\nabla\overline{u}
			+\eta J_2\mbox{div}\overline{\psi}+\eta \overline{J}\mbox{div}\psi_1\\[2mm]
		&+\eta J_2\left((\psi_1)^2-(\psi_2)^2\right)+\eta \overline{J}(\psi_1)^2+2\eta \overline{\psi}\cdot \nabla J_2+2\eta\psi_1\cdot\nabla \overline{J},\\[4mm]
			\end{split}\\
			\begin{split}
			\overline{u}_{t}&+\alpha\mbox{curl}^2\overline{ u}-(2\alpha+\beta)\nabla\mbox{div}\overline{u}\\[2mm]
			=&-u_1\cdot\nabla \overline{u}-\overline{u}\cdot\nabla u_2-\frac{A\gamma}{\gamma-1}\nabla\left((\phi_1)^2-(\phi_2)^2\right)
			+\psi_1 \cdot Q(\overline{u})+\overline{\psi}\cdot Q(u_2)\\[2mm]
			&+\overline{J}\cdot\nabla H_1+J_2\cdot\nabla\overline{ H}-
			\overline{J}\cdot(\nabla H_1)^{T}-J_2\cdot(\nabla\overline{H})^{T}.
			\end{split}
		\end{array}
		\right.
	\end{split}
\end{equation*}
Meanwhile,  $(\overline\phi,\overline\psi,\overline H,\overline{J},\overline{u})$ meets the following initial conditions
\begin{equation*}
	(\overline{\phi},\overline{\psi},\overline{H},\overline{J},\overline{u})|_{t=0}=(0,0,0,0,0)
\end{equation*}
and boundary conditions
\begin{equation*}
	\overline{u}\cdot n|_{\partial\Omega}=0,\quad \mbox{curl}\overline{u}\times n|_{\partial\Omega}=-K(x)\overline{u},
\end{equation*}
\begin{equation*}
	\overline{H}\cdot n|_{\partial\Omega}=0, \quad \mbox{curl}\overline{H}\times n|_{\partial\Omega}=0,
\end{equation*}
\begin{equation*}
	\overline{J}\cdot n|_{\partial\Omega}=0, \quad\mbox{curl}\overline{J}\times n|_{\partial\Omega}=-(\psi_1\cdot n)\overline{J}-(\overline{\psi}\cdot n)J_2.
\end{equation*}

Through the similar argument used in the above six steps, one can ascertain that
\begin{align}
	& G(t)+\int_0^t \left(C_1b_3\|\nabla \overline{u}\|^2+\|\nabla^2 \overline{u}\|^2+\frac{C_2b_1}{2}\|\nabla \overline{H}\|^2
	+b_2\|\nabla \overline{J}\|^2\right) ds\leq \int_0^t\mathcal{C}(s)G(s)ds\nonumber
\end{align} 	
where
\begin{align}
	G= \|\overline{\phi}\|^2_1+\|\overline{\psi}\|^2
	+C_1\|\overline{u}\|^2+\|\nabla\overline{u}\|^2
	+C_2\|\overline{H}\|^2+\|\overline{J}\|^2.\nonumber
\end{align}
Because for $0\leq t\leq T_*$,
\begin{equation*}
		\int_{0}^{t}\mathcal{C}(s)ds\leq C,
\end{equation*}
Utilizing the Gronwall's inequality, we can get the uniqueness follows from $\overline\phi=\overline\psi=\overline H=\overline{J}=\overline{u}=0$.

\qed

\subsection{Proof for Theorem \ref{thm1}}\label{3.6}

With the help of Theorem \ref{thm2}, now we are prepared to establish the local-in-time well-posedness of solution to the original  initial-boundary value problem (\ref{prob1})-(\ref{1.8}),  as shown in Theorem \ref{thm1}.
Let
\begin{align}
J^{*}=J-\frac{H}{\rho}. \nonumber
\end{align}
Then it follows from the equations (\ref{J})-(\ref{initial-2}) that
\begin{eqnarray}{\label{J^*}}
\left\{
\begin{array}{llll}
J^*_t+u\cdot\nabla J^*-\eta\Delta J^*=J^*\cdot\nabla u+\eta J^*\mbox{div}\psi+\eta J^*\psi^2+2\eta \psi\cdot \nabla J^*,\nonumber\\[2mm]
J^*|_{t=0}=0,\nonumber\\[2mm]
J^{*}\cdot n|_{\partial\Omega}=0,\quad \mbox{curl}J^{*}\times n|_{\partial\Omega}=0
\end{array}
\right.
\end{eqnarray}
which, together with a standard energy method, implies that
\begin{align}
J^{*}=0\quad \mbox{for}\quad (x,t)\in  \Omega\times[0,T_{*}],\nonumber
\end{align}
that is $J\equiv\frac{H}{\rho}$.
The rest of the proof analogies that in \cite{Zhu2015}, and we shall omit it here.\qed
\section{Appendix}
In this appendix, we provide the elliptic estimate for $J$ with the specific boundary conditions.
\begin{app}\label{prop-J}
Suppose $J$ satisfies
\begin{equation}\label{4.1}
\left\{
\begin{array}{lll}
-\Delta J=-J_t+F\quad \mbox{in}\quad \Omega\\[2mm]
J\cdot n=0,\quad \mbox{curl}J\times n=-(\psi\times J)\times n=-(\psi\cdot n)J\quad \mbox{on}\quad \partial\Omega.
\end{array}
\right.
\end{equation}
Then it holds
\begin{align}\label{2J}
\|\nabla^2J\|^2
\leq C\left(\|J_t\|^2+\|F\|^2+\| J\|_1^2+\|\psi\|^2_{L^{\infty}} \|J\|^2_1+\|\nabla\psi\|^2_{L^3}\|\nabla J\|^2\right),
\end{align}
and
\begin{align}\label{3J}
\|\nabla^3J\|^2
\leq C\left(\|\nabla J_t\|^2+\|\nabla F\|^2+\|J\|^2_{2}+(\|\psi\|^2_{L^{\infty}}+\|\nabla\psi\|^2_{1})\|J\|^2_{2}\right).
\end{align}
\end{app}
\begin{proof} {\it {Step 1: To estimate $\|\nabla^2J\|$. }}
Firstly, applying Lemma \ref{lem-div-curl}  with $J\cdot n|_{\partial\Omega}=0$, we have
\begin{align}\label{J-0}
\|\nabla^2J\|^2&\leq C\left( \|\mbox{div}J\|_1^2+\|\mbox{curl}J\|_1^2+\|J\|^2\right)\nonumber\\[2mm]
&\leq \|\nabla\mbox{div}J\|^2+\|\mbox{curl}^2J\|^2+|\mbox{curl}J\times n|^2_{H^{\frac{1}{2}}(\partial\Omega)}+\|\nabla J\|^2+\|J\|^2,
\end{align}
where the boundary integral can be controlled as
\begin{align}\label{J-0-1}
|\mbox{curl}J\times n|_{H^{\frac{1}{2}}(\partial\Omega)}&=|\psi\cdot n J|_{H^{\frac{1}{2}}(\partial\Omega)}\nonumber\\[2mm]
&\leq C\|\psi\cdot n J\|_{1}\nonumber\\[2mm]
&\leq C\left(\|\psi\|^2_{L^{\infty}} \|J\|^2_1+\|\nabla\psi\|^2_{L^3}\|\nabla J\|^2\right).
\end{align}
Next, one can rewrite the equation $(\ref{4.1})_1$ as
\begin{align}\label{equation-J}
\mbox{curl}^2J-\nabla\mbox{div}J=-J_t+F.
\end{align}
Multiplying  \eqref{equation-J} by $\nabla\mbox{div}J$, it holds
\begin{align}\label{J-1}
\|\nabla\mbox{div}J\|^2=\int_{\Omega}\mbox{curl}^2J\cdot\nabla\mbox{div}J+\int_{\Omega}(J_t-F)\cdot\nabla\mbox{div}J.
\end{align}
Integrating by parts with the boundary condition $(\mbox{curl}J \times n)|_{\partial\Omega}=-(\psi\cdot n)J$,  it leads to
\begin{align}\label{J-3}
\int_{\Omega}\mbox{curl}^2J\cdot\nabla\mbox{div}J&=-\int_{\partial\Omega}(\mbox{curl}J\times n)\cdot\nabla\mbox{div}J\nonumber\\[2mm]
&\leq C\|\nabla(\psi\cdot nJ)\|\cdot \|\nabla\mbox{div}J\|\nonumber\\[2mm]
&\leq \varepsilon \|\nabla\mbox{div}J\|^2+C\left( \|\psi\|^2_{L^{\infty}}+ \|\nabla\psi\|^2_{L^3}\right)\|\nabla J\|^2,
\end{align}
where we have used  Proposition \ref{lem-boun} with $v=(\psi\cdot n)J$  and $f=\mbox{div}J$.
Therefore, putting  \eqref{J-3} into \eqref{J-1} and selecting $\varepsilon$ small enough, we can conclude that
\begin{align}\label{div}
\|\nabla\mbox{div}J\|^2\leq C\left( \|\psi\|^2_{L^{\infty}}+ \|\nabla\psi\|^2_{L^3}\right)\|\nabla J\|^2+C(\|J_t\|^2+\|F\|^2).
\end{align}
Moreover, multiplying \eqref{equation-J} by $\mbox{curl}^2J$,  and using \eqref{J-3} again, we have
\begin{align*}
\|\mbox{curl}^2J\|^2&=\int_{\Omega}\mbox{curl}^2J\cdot\nabla\mbox{div}J
+\int_{\Omega}(-J_t+F)\cdot\mbox{curl}^2J\nonumber\\[2mm]
&\leq \epsilon \|\nabla\mbox{div}J\|^2+C\left( \|\psi\|^2_{L^{\infty}}+ \|\nabla\psi\|^2_{L^3}\right)\|\nabla J\|^2+\frac{1}{4}\|\mbox{curl}^2J\|^2+C(\|J_t\|^2+\|F\|^2),
\end{align*}
which yields
\begin{align}\label{curl}
\|\mbox{curl}^2J\|^2
&\leq \epsilon \|\nabla\mbox{div}J\|^2+C\left( \|\psi\|^2_{L^{\infty}}+ \|\nabla\psi\|^2_{L^3}\right)\|\nabla J\|^2+C(\|J_t\|^2+\|F\|^2).
\end{align}
Adding (\ref{div}) and (\ref{curl}) together, and choosing $\epsilon$ small enough yields that
\begin{align}
\|\mbox{curl}^2J\|^2+\|\nabla\mbox{div}J\|^2 \leq C\left( \|\psi\|^2_{L^{\infty}}+ \|\nabla\psi\|^2_{L^3}\right)\|\nabla J\|^2+C\left(\|J_t\|^2+\|F\|^2\right),\nonumber
\end{align}
which combined with (\ref{J-0}) and  (\ref{J-0-1})  immediately implies the desired estimate \eqref{2J}.\\

 {\it {Step 2: To estimate $\|\nabla^3J\|$. }}
Applying Lemma \ref{lem-div-curl} twice, one obtians
\begin{align}\label{0}
&\|\nabla^3J\|\leq C\left(\|\mbox{div}J\|_{H^2}+\|\mbox{curl}J\|_{H^2}+\|J\|\right)\nonumber\\[2mm]
&\leq C\left(\|\nabla^2\mbox{div}J\|+\|\mbox{curl}^3J\|+\|J\|_{2}+|\mbox{curl}J\times n|_{H^{\frac{3}{2}}(\partial\Omega)}+|\mbox{curl}^2J\cdot n|_{H^{\frac{1}{2}}(\partial\Omega)}\right).
\end{align}
By applying the boundary condition on $\mbox{curl}J$ directly, it holds
\begin{align}\label{b1}
|\mbox{curl}J\times n|_{H^{\frac{3}{2}}(\partial\Omega)}&=|\psi\cdot n J|_{H^{\frac{3}{2}}(\partial\Omega)}
\leq C\|\psi\cdot J\|_{2}\nonumber\\[2mm]
&\leq C(\|\psi\|_{L^{\infty}}+\|\nabla\psi\|_{1})\|J\|_{2}.
\end{align}
Remark that the boundary condition $\mbox{curl}J\times n=-(\psi\times J)\times n$ implies $$\mbox{curl}(\mbox{curl}J+\psi\times J)\cdot n=0 \quad \mbox{on}\quad  \partial\Omega,$$ so
\begin{equation*}
\mbox{curl}^2J\cdot n=-\mbox{curl}(\psi\times J)\cdot n\quad \mbox{on}\quad \partial\Omega.
\end{equation*}
Therefore, it gives
\begin{align}\label{b2}
|\mbox{curl}^2J\cdot n|_{H^{\frac{1}{2}}(\partial\Omega)}&=|\mbox{curl}(\psi\times J)\cdot n|_{H^{\frac{1}{2}}(\partial\Omega)}
\leq C\|\nabla(\psi\cdot J)\|_{1}\nonumber\\[2mm]
&\leq C(\|\psi\|_{L^{\infty}}+\|\nabla\psi\|_{1})\|J\|_{2}.
\end{align}
Putting \eqref{b1} and \eqref{b2} into \eqref{0} yields
\begin{equation}\label{3J-0}
\|\nabla^3J\|^2\leq C\left(\|\nabla^2\mbox{div}J\|^2+\|\mbox{curl}^3J\|^2\right)+C(1+\|\psi\|^2_{L^{\infty}}+\|\nabla\psi\|^2_{1})\|J\|^2_{2}.
\end{equation}

Next, taking the operator $\mbox{curl}$ to equation \eqref{equation-J}, and then multiplying  the resulted equation  by $\mbox{curl}^3 J$, one can obtain
\begin{align*}
\|\mbox{curl}^3 J \|^2 = \int_{\Omega} (-\mbox{curl}J_t+\mbox{curl}F) \cdot \mbox{curl}^3 J
  \leq \frac{1}{2}\|\mbox{curl}^3 J \|^2 + C (\|\nabla J_t\|^2+\|\nabla F\|^2),
\end{align*}
which implies that
\begin{align}\label{curl-3}
\|\mbox{curl}^3 J \|^2  \leq C (\|\nabla J_t\|^2+\|\nabla F\|^2).
\end{align}
Furthermore, computing  $\int_{\Omega}\nabla_{i} (\ref{equation-J})^j\cdot \nabla_{ij}\mbox{div}J$, we have
\begin{align}\label{3J-1}
\|\nabla^2\mbox{div} J \|^2& = \int_{\Omega} \nabla\mbox{curl}^2J\cdot\nabla^2\mbox{div}J+\int_{\Omega} (\nabla J_t-\nabla F)\cdot\nabla^2\mbox{div}J\nonumber\\[2mm]
&\leq \frac{1}{2}\|\nabla^2\mbox{div} J \|^2+C\|\nabla \mbox{curl}^2 J \|^2+\|\nabla J_t\|^2+\|\nabla F\|^2.
\end{align}
Applying Lemma \ref{lem-div-curl}  to $\mbox{curl}^2 $ and recalling \eqref{b2}, we acquire
\begin{align}
\| \nabla\mbox{curl}^2J\|^2&\leq C(\|\mbox{curl}^3J\|^2+|\mbox{curl}^2J\cdot n|^2_{H^{\frac{1}{2}}(\partial\Omega)}+\|\mbox{curl}^2J\|^2)\nonumber\\[2mm]
&\leq  C(\|\mbox{curl}^3J\|^2+\|J\|^2_{2}).\nonumber
\end{align}
Taking  the above inequality back  into (\ref{3J-1}), it holds
\begin{align}\label{div-3}
\|\nabla^2\mbox{div} J \|^2\leq C\|\mbox{curl}^3 J \|^2+\|J\|_{2}^2+\|\nabla J_t\|^2+\|\nabla F\|^2.
\end{align}
Then, by adding \eqref{curl-3} and \eqref{div-3}, and making use of \eqref{3J-0}, the desired estimate \eqref{3J} has been obtained. The proof is completed.

\end{proof}

\centerline{\bf Acknowledgements}
Liu's research is supported by National Natural Science Foundation of China (No. 12471198, No.
12431018). Luo's research is supported by a grant from the Research Grants Council of the Hong Kong Special Administrative Region, China (Project No. 11306621). Zhong is supported by National Natural Science Foundation of China (No. 12201520).

\end{document}